\documentclass[13pt]{amsart}
\usepackage{amsmath,amssymb,amsthm,amscd}
\usepackage{amsmath}
\usepackage{}
\usepackage{amsfonts}
\usepackage{amsfonts,latexsym,rawfonts,amsmath,amssymb,amsthm}

\usepackage{pdfsync}

\usepackage{graphicx}
\usepackage{tikz}
\numberwithin{equation}{section}

\newtheorem{prop}{Proposition}[section]
\newtheorem{theo}[prop]{Theorem}
\newtheorem{lemm}[prop]{Lemma}

\newtheorem{rem}[prop]{Remark}

\newtheorem{defi}[prop]{Definition}

\newtheorem{claim}[prop]{Claim}
\def\begeq{\begin{equation}}
\def\endeq{\end{equation}}

\pdfoutput=1

\begin{document}

\title{Singular K\"ahler-Einstein metrics on  $\mathbb Q$-Fano compactifications of   Lie groups}
\author{Yan $\text{Li}^{*}$, Gang $\text{Tian}^{\dag}$ and Xiaohua $\text{Zhu}^{\ddag}$}

\address{BICMR and SMS, Peking University, Beijing 100871, China.}
\email{liyanmath@pku.edu.cn,\ \ \ tian@math.pku.edu.cn\\\ xhzhu@math.pku.edu.cn}

\thanks {$^*$Partially supported by China Post-Doctoral Grant BX20180010.}
\thanks {$^{\dag}$Partially supported by NSFC grants 11331001 and 11890661.}
\thanks {$^{\ddag}$Partially supported by NSFC Grants  11771019 and BJSF Grants Z180004.}
\subjclass[2000]{Primary: 53C25; Secondary:
32Q20, 58D25, 14L10}

\keywords{K\"ahler-Einstein metrics, $\mathbb Q$-Fano compactifications of Lie groups, moment polytopes, reduced Ding functional}

\begin{abstract}
In this paper, we prove an existence result for K\"ahler-Einstein metrics on $\mathbb Q$-Fano compactifications of Lie groups.  As an application, we classify $\mathbb Q$-Fano compactifications of $SO_4(\mathbb C)$ which admit a K\"ahler-Einstein metric with the same volume as that of a smooth Fano compactification of $SO_4(\mathbb C)$.
\end{abstract}
\maketitle

\tableofcontents

\section{Introduction}

Let $G$ be an $n$-dimensional connected, complex reductive Lie group which is the complexification of a compact Lie group $K$.  Let  $T^{\mathbb C}$ be  a maximal Cartan torus of $G$ whose dimension is $r$.  Denote by $\Phi_+$ a positive roots system  associated to $T^{\mathbb C}$. Put
 \begin{align}\label{rho}\rho\,=\,{\frac 1 2}\sum_{\alpha\in\Phi_+}\alpha.
 \end{align}
It can be regarded as a character in $\mathfrak a^*$, where $\mathfrak a^*$ is the dual space  of  real part $\mathfrak a$ of Lie algebra of $T^{\mathbb C}$.
Let $\pi$ be a function on ${\mathfrak a^*}$ defined by
$$\pi(y)=\prod_{\alpha\in\Phi_+}\langle\alpha,y\rangle^2, ~y\in {\mathfrak a^*},$$
where $ \langle\cdot,\cdot \rangle$ \footnote{Without of confusion, we also write $\langle\alpha,y\rangle$ as $\alpha(y)$ for simplicity.} denotes the Cartan-Killing inner product on  ${\mathfrak a^*}$.

Let $M$  be a $\mathbb Q$-Fano compactification of $G$.
Since $M$ contains a closure $Z$ of $T^{\mathbb C}$-orbit,   there is an associated moment polytope $P$ of $Z$ induced by $(M, -K_M)$ \cite{AB1, AB2}.
Let $P_+$ be the positive part of $P$ defined by
$$P_+\,=\,\{y\in P|~ \langle\alpha,y\rangle >0, ~\forall~ \alpha\in \Phi_+\}.$$
Denote by $2P_+$ a dilation of $P_+$ at rate $2$.  We define the barycenter of $2P_+$ with respect to the weighted measure $\pi(y)dy$ by
$$bar(2P_+)\,=\,\frac{\int_{2P_+}y\pi(y) \,dy}{\int_{2P_+}\pi(y) \,dy}.$$

In \cite{Del2}, Delcroix proved the following the existence theorem for K\"ahler-Einstein metrics on smooth Fano compactifications of $G$.

\begin{theo}\label{de} Let $M$ be a smooth Fano $G$-compactification. Then $M$ admits a K\"ahler-Einstein metric if and only if
\begin{align}\label{bary}
bar(2P_+)\,\in\, 4\rho+\Xi,
\end{align}
where $\Xi$ is the relative interior of the cone generated by $\Phi_+$.
\end{theo}

Another proof of Theorem \ref{de} was given by Li, Zhou and Zhu  \cite{LZZ}. They also showed  that (\ref{bary}) is actually equivalent to the $K$-stability condition in terms of \cite {Ti97} and \cite{Do02} by constructing $\mathbb C^*$-action through piecewisely rationally linear function which is invariant under the Weyl group action. In particular, it implies that $M$ is $K$-unstable if $bar(2P_+) \not\in \overline { 4\rho+\Xi}$.
A more general construction of $\mathbb C^*$-action was also discussed in \cite{Del3}.

In the present paper, we extend the above theorem to $\mathbb Q$-Fano compactifications  of $G$ which may be singular.
It is well known that any $\mathbb Q$-Fano compactification of $G$ has  klt-singularities \cite{AK}.
For a $\mathbb Q$-Fano variety  $M$ with klt-singularities, there is naturally a class of admissible K\"ahler metrics induced by the Fubini-Study metric (cf. \cite{DT}).
In \cite{BBEGZ}, Berman, Boucksom, Eyssidieux, Guedj and Zeriahi introduce a class of K\"ahler potentials associated to admissible K\"ahler metrics and refer it as the $\mathcal E^1(M,-K_M)$ space.  Then they define the singular K\"ahler-Einstein metric on $M$ with the  K\"ahler potential in $\mathcal E^1(M,-K_M)$ via the complex Monge-Amp\`ere equation, which is the usual K\"ahler-Einstein metric on the smooth  part of $M$.  It is an natural problem to
establish an extension of the Yau-Tian-Donaldson conjecture we have solved for smooth Fano manifolds
\cite{Ti97, Ti15}, that is,
an equivalence relation between the existence of such singular K\"ahler-Einstein metrics and the $K$-stability on  a $\mathbb Q$-Fano variety $M$ with
klt-singularities.
There are many recent works on this fundamental problem. We refer the readers to \cite{Berman-Berndtsson, BBJ15, LTW17, LTW19, Li19}, etc..

In this paper, we will assume that the moment polytope $P$ of $Z$ is fine in sense of \cite{Do}, namely, each vertex of $P$ is the intersection of precisely $r$ facets. We will prove

\begin{theo}\label{LTZ2}  Let $M$ be a $\mathbb Q$-Fano compactification of $G$ such that the moment polytope $P$  of $Z$ is fine.
Then $M$ admits a singular K\"ahler-Einstein metric if and only if (\ref{bary}) holds.
\end{theo}

By a result of  Abreu  \cite{Ab-JDG},  the polytope $P$ of $Z$ being fine is equivalent to that the metric induced by the Guillemin function can be extended
to a  K\"ahler orbifold metric on $Z$. \footnote{It can not be guaranteed that the $G$--compactification is smooth even if  $Z$ is smooth \cite{AK}.}
It follows from the fineness assumption of $P$ in Theorem \ref{LTZ2} that the Guillemin function of $2P$ induces a $K\times K$-invariant singular metric $\omega_{2P}$ in  $\mathcal E^1(M,-K_M)$ (cf. Lemma \ref{admissible}). Moreover, we can prove that the Ricci potential of $\omega_{2P}$ on $M$ is uniformly bounded   above. We note that $P$ is always fine when rank$(G)=2$ \cite[Chapter 3]{Fulton}.  Thus  for a $\mathbb Q$-Fano compactification of $G$ with rank$(G)=2$, $M$ admits a singular K\"ahler-Einstein metric if and only if  \eqref{bary} holds.  As an application of Theorem \ref{LTZ2},
we show that there is only one example of
non-smooth Gorenstein Fano  $SO_4(\mathbb C)$-compactifications which admits a singular K\"ahler-Einstein metric (cf. Section 7.1).

On the other hand, it has been shown in \cite{Del2} and \cite{LTZ} that there are only three  smooth Fano compactifications  of $SO_4(\mathbb C)$,  i.e., \emph{Case-1.1.2}, \emph{Case-1.2.1} and \emph{Case-2} in Section \ref{sect-6-1}.
The first two manifolds do not admit any K\"ahler-Einstein metric. By Theorem \ref{LTZ2}, we further  prove

\begin{theo}\label{LTZ3}
There is no $\mathbb Q$-Fano compactification of $SO_4(\mathbb C)$ which admits a singular K\"ahler-Einstein metric with  the same volume  as Case-1.1.2 or {Case-1.2.1} in  Section 7.1.
\end{theo}

Theorem \ref{LTZ3} gives a partial answer to a question proposed in \cite{LTZ} about limit of K\"ahler-Ricci flow on either \emph{Case-1.1.2} or  \emph{Case-1.2.1}.  It has been proved there that
the flow  has type II  singularities on each of    \emph{Case-1.1.2} and    \emph{Case-1.2.1} whenever the initial metric is $K\times K$-invariant. By the Hamilton-Tian conjecture \cite{Ti97, Bam, CW},  the limit should be a $\mathbb Q$-Fano variety with a singular K\"ahler-Ricci soliton of the same volume as that of initial metric.  However, by Theorem \ref{LTZ3}, the limit can not
be a $\mathbb Q$-Fano compactification of $SO_4(\mathbb C)$ with a singular K\"ahler-Einstein metric. This implies that the limiting soliton
will has less homogeneity than the initial one, which is totally different from the situation of smooth convergence of $K\times K$-metrics on a smooth compactification of  Lie group \cite{LTZ}.

As in \cite{BBEGZ}, we use the variation method to prove Theorem \ref{LTZ2}, more precisely, we will prove that a modified version of the
Ding functional $\mathcal D(\cdot)$ is proper under the condition (\ref{bary}).  This functional is defined for a class of convex functions
 $\mathcal E^1_{K\times K}(2P)$ associated to $K\times K$-invariant metrics on the orbit of $G$ (cf. Section 4, 6).
The key point is that the Ricci potential $h_0$ of the Guillemin metric $\omega_{2P}$ is bounded from above when $P$ is fine (cf. Proposition \ref{ricci-potential}).  This enables us to control the nonlinear part $\mathcal F(\cdot)$ of $\mathcal D(\cdot)$ by modifying $\mathcal D(\cdot)$ as done in  \cite {Do02, LZ} (cf. Section 6.1). We shall note that it is in general impossible to get a lower bound of $h_0$ if the compactification is a singular variety (cf. Remark \ref{h-remark}).
On the other hand, we expect that the ``fine" condition in Theorem \ref{LTZ2} can be dropped.

The minimizer of $\mathcal D(\cdot)$ corresponds to a singular K\"ahler-Einstein metric.
We will prove  the semi-continuity of $\mathcal D(\cdot)$ and derive the K\"ahler-Einstein equation for the minimizer (cf. Proposition \ref{minimizer-in-E-prop}).  Our proof is similar to what Berman and Berndtsson studied on toric varieties in \cite{Berman-Berndtsson}.

The proof of the necessity part of Theorem \ref{LTZ2} is same as one in Theorem \ref{de}. In fact,  a $\mathbb Q$-Fano compactification of $G$  is $K$-unstable if $bar(2P_+) \not\in \overline {4\rho+\Xi}$   \cite{LZZ}.  This will be a contradiction to the semi-stability of $\mathbb Q$-Fano variety with a singular  K\"ahler-Einstein metric (cf. \cite{LTW17}). We will omit this part.

The organization of paper is as follows.  In Section 2, we  recall some notations  in  \cite{BBEGZ} for singular K\"ahler-Einstein metrics on
$\mathbb Q$-Fano varieties. In Section 3, we  introduce   a subspace    $\mathcal E^1_{K\times K}(M,-K_M)$ of
 $\mathcal E^1(M,-K_M)$ and prove that the Guillemin function lies in this space (cf. Lemma \ref{admissible}).  In Section 4, we prove that
$\mathcal E^1_{K\times K}(M,-K_M)$ is equivalent to  a dual space  $\mathcal E^1_{K\times K}(2P)$ of Legendre functions (cf. Theorem \ref{E1-legendre}). In Section 5,  we compute   the Ricci potential $h_0$ of $\omega_{2P}$ and  show that it is bounded from above (cf. Proposition \ref{ricci-potential}).
The sufficient part of Theorem \ref{LTZ2} will be proved in Section 6.  In Section 7,  we construct many $\mathbb Q$-Fano compactifications of ${SO}_4(\mathbb C)$ and in particular, we will prove Theorem \ref{LTZ3}.

\section{Preliminary on  $\mathbb Q$-Fano  varieties}

For a $\mathbb Q$-Fano variety $M$, by Kodaira's embedding theorem, there is an integer $\ell > 0$ such that we can embed $M$ into a projective space $\mathbb {CP}^N$ by a basis of $H^0(M, K_M^{-\ell})$, for simplicity, we assume $M\subset \mathbb {CP}^N$. Then we have a metric
$$\omega_0\,=\,\frac{1}{\ell}\,\omega_{FS}|_{M}\,\in \,2\pi c_1(M),$$
where $\omega_{FS}$ is the Fubini-Study metric of $\mathbb {CP}^N$.
Moreover, there is a Ricci potential $h_0$ of $\omega_0$ such that
$${\rm Ric}(\omega_0)-\omega_0\,=\,\sqrt{-1}\partial\bar\partial h_0, ~{\rm on} ~M_{\rm reg}.$$
In the case that $M$ has only
klt-singularities, $e^{h_0}$ is $L^p$-integrate for some $p>1$ (cf. \cite{DT, BBEGZ}). We call $\omega$ an admissible K\"ahler metric on $M$ if there are an embedding $M\subset \mathbb {CP}^N$ as above and a smooth function $\varphi$ on $ \mathbb {CP}^N$ satisfying:
$$\omega\,=\,\omega_{\varphi}|_{ \mathbb {CP}^N}\,:=\,\omega_0+\sqrt{-1}\,\partial\bar\partial \varphi.$$
In particular, $\varphi$ is a function on $M$  with $\varphi\in L^\infty(M)\cap C^\infty(M_{\text{reg}})$, called an admissible K\"ahler potential associated to $\omega_0$. \footnote{For simplicity, we will denote a K\"ahler metric by its K\"ahler form thereafter.}

For a general (possibly unbounded) K\"ahler potential $\varphi$,  we define its complex Monge-Amp\`ere measure $\omega_{\varphi}^n$ by
$$\omega_{\varphi}^n\,=\,\lim_{j\to \infty}\,\omega_{\varphi_j}^n,$$
where $\varphi_j\,=\,{\rm max}\{\varphi,-j\}$.
According to \cite{BBEGZ}, we say that $\varphi$ (or $\omega_{\varphi}^n$)  has  full Monge-Amp\'ere (MA) mass if
$$\int_M \omega_{\varphi}^n\,=\,\int_M \omega_0^n.$$
The MA-measure $\omega_{\varphi}^n$ with a full MA-mass has no mass on the pluripolar set of $\varphi$ in $M$. Thus we need to consider the
measure on  $M_{\rm reg}$.  Moreover, $e^{-\varphi}$ is $L^p$-integrable for any $p>0$ associated to $\omega_0^n$.

\begin{defi}\label{singular-ke} We call $\omega_{\varphi}$ a (singular) K\"ahler-Einstein metric on $M$ with full MA-mass if
$\varphi$ satisfies the following  complex Monge-Amp\'ere equation,
\begin{align}\label{singular-ke-equation}
\omega_{\varphi}^n\,=\,e^{h_0-\varphi}\omega_0^n.
\end{align}
\end{defi}

It has been shown in \cite{BBEGZ} that $\varphi$ is $C^\infty$ on $M_{\rm reg}$ if it is a solution of  (\ref{singular-ke-equation}).  Thus $\omega_{\varphi}$ satisfies the K\"ahler-Einstein equation on $M_{\rm reg}$,
$${\rm Ric}(\omega_{\varphi})\,=\,\omega_{\varphi}.$$

\subsection{The space $\mathcal E^1(M,-K_M)$ and the Ding functional}

On a smooth Fano manifold, there is a well-known Euler-Langrange functional for K\"ahler potentials associated to (\ref{singular-ke-equation}),  often referred as the Ding functional or F-functional, defined by (cf. \cite{Di88, Ti87})
\begin{align}\label{ding-functional}
 F(\phi)\,=\,-\frac1{(n+1)V}\sum_{k=0}^n\,\int_M\phi\omega_\phi^k\wedge\omega_0^{n-k}-\log\left(\frac1V\int_Me^{h_0-\phi}\omega_0^n\right).
\end{align}
In case of $\mathbb Q$-Fano manifold with klt-singularities, Berman, Boucksom, Eyssidieux, Guedj and Zeriahi \cite{BBEGZ} extended $F(\cdot)$ to
the space $\mathcal E^1(M,-K_M)$ defined by
\begin{align}\mathcal E^1(M,-K_M)\,=\,\{\phi| ~&\phi ~{\rm has  ~a ~full ~MA~mass~and}  \notag\\
& \sup_M\phi =0,~ I(\phi)=\int_M-\phi\omega_\phi^n<\infty\}.\notag
\end{align}
They showed that $\mathcal E^1(M,-K_M)$ is compact in certain weak topology. By a result of Davas \cite{DT},
$\mathcal E^1(M,-K_M)$ is in fact compact in the topology of $L^1$-distance. It provides a variational approach to study (\ref{singular-ke-equation}).

\begin{defi}\label{properdef}\cite{Ti97,BBEGZ}  The functional $F(\cdot)$ is called proper if there is a continuous function $p(t)$ on $\mathbb R$ with the property
$\lim_{t\to+\infty} p(t)\,=\,+\infty$, such that
\begin{equation}\label{proper}
F(\phi)\,\ge\, p( I(\phi)),~\forall \phi  \in \mathcal E^1(M,-K_M).
\end{equation}

\end{defi}
In \cite{BBEGZ}, Berman, Boucksom, Eyssidieux, Guedj and Zeriahi proved the existence of solutions for (\ref{singular-ke-equation}) under the properness assumption (\ref{proper}) of $F(\cdot)$.  However, this assumption does not hold in the case of existence
of non-zero holomorphic vector fields such as in our case of $\mathbb Q$-Fano $G$-compactifications. So we need to consider the reduced Ding functional instead to overcome this new difficulty as done on toric varieties \cite{Berman-Berndtsson, LZ}.

\section{Moment polytope and $K\times K$-invariant metrics}

Let $M$ be a $\mathbb Q$-Fano compactification of $G$ with $Z$ being the closure of a maximal complex torus $T^{\mathbb C}$-orbit.
We  first  characterize the moment polytope $P$  of $Z$ associated to $(M, K_M^{-1})$.  Let $\{F_A\}_{A=1,...,d_0}$ be the facets of $P$ and $\{F_A\}_{A=1,...,d_+}$ be those whose interior intersects $\mathfrak a_+^*$. Suppose that
\begin{align}\label{polytope-eq}
P=\cap_{A=1}^{d_0}\{l_A^o:=\lambda_{A}-u_A(y)\geq0\}
\end{align}
for some prime vector $u_A\in\mathfrak N$ and the facet
$$F_A\subseteq\{l^o_{A}=0\},A=1,...,d_0.$$
Let $W$ be the Weyl group of $(G,T^{\mathbb C})$. By  the $W$-invariance, for each $A\in\{1,...,d_0\}$, there is some $w_A\in W$ such that $w_A(F_A)\in\{F_B\}_{B=1,...,d_+}$. Denote by $\rho_A\,=\,w_A^{-1}(\rho)$, where $\rho\in\mathfrak a_+^*$ is given by (\ref{rho}). Then $\rho_A(u_A)$ is independent of the choice of $w_A\in W$ and hence it is well-defined.

The following is due to  \cite{Brion}.

\begin{lemm}\label{polytope-coefficient}
Let $M$ be a $\mathbb Q$-Fano compactification of $G$ with $P$ being the associated moment polytope. Then for each $A=1,...,d_0$, it holds
\begin{align}\label{relation-lambda}\lambda_{A}\,=\,1+2\rho_A(u_A).
\end{align}
\end{lemm}

\begin{proof}
Suppose that $-mK_M$ is a Cartier divisor for some  $m\in\mathbb N$. Then by \cite[Section 3]{Brion},
$$-mK_M|_Z\,=\,\sum_Am(1+2\rho_A(u_A))D_A,$$
where $D_A$ is the toric divisor of $Z$ associated to $u_A$. Thus the associated polytope of
 $(Z,-mK_M|_Z)$ is given by
$$P(Z,-mK_M|_Z)\,=\,\cap_{A=1}^{d_0}\{m(1+2\rho_A(u_A))-u_A(y)\geq0\},$$
which is precisely $mP$.  Thus (\ref{relation-lambda}) is true.
\end{proof}

\subsection{$K\times K$-invariant metrics }

On a $\mathbb Q$-Fano compactification of $G$, we may regard the $G\times G$ action as a subgroup of $PGL_{N+1}(\mathbb C)$ which acts holomorphically on the hyperplane bundle $L\,=\,{\mathcal O}_{\mathbb CP^N}(-1)$. Then any admissible $K\times K$-invariant K\"ahler metric
 $\omega_\phi\in\frac{ 2\pi }{\ell} c_1(L)$  can be regarded as a restriction of $K\times K$-invariant K\"ahler metric of $\mathbb  CP^N$. Thus the moment polytope $P$ associated to $(Z, L|_Z)$
is a $W$-invariant rational polytope in $\mathfrak a^*$.
By the $K\times K$-invariance, the restriction of $\omega_\phi$ on $T^{\mathbb C}$ is an open toric K\"ahler metric.
Hence, it induces  a strictly convex, $W$-invariant function $\psi_\phi$ on ${\mathfrak a}$  \cite{AL} (also see Lemma \ref{Hessian} below)  such that
\begin{align}\label{convex-potential}\omega_\phi\,=\,\sqrt{-1}\partial\bar\partial \psi_\phi, ~{\rm on}~ T^{\mathbb C}.
\end{align}

By the $KAK$-decomposition (\cite[Theorem 7.39]{Kna2}), for any $g\in G$,
there are $k_1,\,k_2\in K$ and $x\in\mathfrak a$ such that $g=k_1\exp(x)k_2$. Here $x$ is uniquely determined up to a $W$-action. This means that $x$ is unique in $\overline{\mathfrak a_+}$.
Thus there is a bijection between $K\times K$-invariant functions $\Psi$ on $G$ and $W$-invariant functions  $\psi$ on
$\mathfrak a$ which is given  by
$$\Psi( \exp(\cdot))\,=\,\psi(\cdot):~{\mathfrak a}\to\mathbb R.$$
Clearly, when a $W$-invariant $\psi$ is given, $\Psi$ is well-defined.  Without of  confusion,  we will not distinguish $\psi$ and $\Psi$,  and  call $\Psi$ convex on $G$ if  $\psi$  is convex on ${\mathfrak a}$.

The following $KAK$-integral formula can be found in \cite[Proposition 5.28]{Kna2}.

\begin{prop}\label{KAK int}
Let $dV_G$ be a Haar measure on $G$ and $dx$ the Lebesgue measure
on $\mathfrak{a}$.
Then there exists a constant $C_H>0$ such that for any
$K\times K$-invariant, $dV_G$-integrable function $\psi$ on $G$,
$$\int_G \psi(g)\,dV_G\,=\, C_H\,\int_{\mathfrak{a}_+}\psi(x)\mathbf{J}(x)\,dx,$$
where
$$\mathbf J(x)\,=\,\prod_{\alpha \in \Phi_+} \sinh^2\alpha(x).$$
\end{prop}

Without loss of generality, we may normalize $C_H=1$ for simplicity.

Next we recall a local holomorphic coordinate system on $G$ used in \cite{Del2}. By the standard Cartan decomposition, we can decompose
$\mathfrak g$ as
$$\mathfrak g\,=\,\left(\mathfrak t\oplus\mathfrak a\right)\oplus\left(\oplus_{\alpha\in\Phi}V_{\alpha}\right),$$
where $\mathfrak t$ is the Lie algebra of  $T$ and
 $$V_{\alpha}\,=\,\{X\in\mathfrak g|~ad_H(X)\,=\,\alpha(H)X,~\forall H\in\mathfrak t\oplus\mathfrak a\}$$
 is the root space of complex dimension $1$ with respect to $\alpha$.  By \cite{Hel}, one can choose $X_{\alpha}\in V_{\alpha}$ such that $X_{-\alpha}=-\iota(X_{\alpha})$ and
$[X_{\alpha},X_{-\alpha}]=\alpha^{\vee},$ where $\iota$ is the Cartan involution and $\alpha^{\vee}$ is the dual of $\alpha$ by the Killing form.
Let $E_{\alpha}=X_{\alpha}-X_{-\alpha}$ and $E_{-\alpha}=J(X_{\alpha}+X_{-\alpha})$. Denoted by $\mathfrak k_{\alpha},\,\mathfrak k_{-\alpha}$ the real line spanned by $E_\alpha,\,E_{-\alpha}$, respectively.
Then we get  the Cartan decomposition of  Lie algebra $\mathfrak k$ of $K$ as follows,
$$\mathfrak k=\mathfrak t\oplus\left(\oplus_{\alpha\in\Phi_+}\left(\mathfrak k_{\alpha}\oplus\mathfrak k_{-\alpha}\right)\right).$$

 Choose a real basis $\{E^0_1,...,E^0_r\}$ of $\mathfrak t$, where  $r$ is the dimension of $T$.  Then $\{E^0_1,...,E^0_r\}$ together with $\{E_{\alpha},E_{-\alpha}\}_{\alpha\in\Phi_+}$ forms a real basis of $\mathfrak k$, which is indexed by $\{E_1,...,E_n\}$. We can also regard $\{E_1,...,E_n\}$
 as a complex basis of $\mathfrak g$. For any $g\in G$, we define local coordinates $\{z_{(g)}^i\}_{i=1,...,n}$ on a neighborhood of $g$ by
$$(z_{(g)}^i)\to\exp(z_{(g)}^i E_i)g.$$
It is easy to see that $\theta^i|_g\,=\,dz_{(g)}^i|_g$,  where the dual $\theta^i$ of $E_i$  is a right-invariant holomorphic $1$-form.
Thus $\displaystyle{\wedge_{i=1}^n\left(dz_{(g)}^i\wedge d\bar{z}_{(g)}^i\right)}|_g$ is also a right-invariant $(n,n)$-form,
which defines a Haar measure $dV_G$.

For a $K\times K$-invariant function $\psi$, Delcroix computed the Hassian of $\psi$ in the above local coordinates as follows \cite[Theorem 1.2]{Del2}.

\begin{lemm}\label{Hessian}
Let $\psi$ be a $K\times K$ invariant function on $G$.
Then  for any $x\in \mathfrak{a}_+$,
the complex Hessian matrix of $\psi$  in the  above coordinates is diagonal by blocks, and equals to
\begin{equation}\label{+21}
\mathrm{Hess}_{\mathbb{C}}(\psi)(\exp(x)) =
\begin{pmatrix}
\frac{1}{4}\mathrm{Hess}_{\mathbb{R}}(\psi)(x)& 0 &  & & 0 \\
 0 & M_{\alpha_{(1)}}(x) & & & 0 \\
 0 & 0 & \ddots & & \vdots \\
\vdots & \vdots & & \ddots & 0\\
 0 & 0 &  & & M_{\alpha_{(\frac{n-r}{2})}}(x)\\
\end{pmatrix},
\end{equation}
where $\Phi_+=\{\alpha_{(1)},...,\alpha_{(\frac{n-r}{2})}\}$ is the set of positive roots and
\[
M_{\alpha_{(i)}}(x) = \frac{1}{2}\alpha_{(i)}(\partial \psi(x))
\begin{pmatrix}
\coth\alpha_{(i)}(x) & \sqrt{-1} \\
-\sqrt{-1} & \coth\alpha_{(i)}(x) \\
\end{pmatrix}.
\]
\end{lemm}

By (\ref{+21}) in Lemma \ref{Hessian},
we see that $\psi_\phi$ induced by an admissible $K\times K$-invariant K\"ahler form $\omega_\phi$ is convex on $\mathfrak{a}$.
The complex Monge-Amp\'ere measure is given by
$$\omega^n\,=\,(\sqrt{-1}\partial\bar{\partial}\psi_\phi)^n\,=\,\mathrm{MA}_{\mathbb C}(\psi_\phi)\,dV_G,$$
where
\begin{equation}\label{MA}
\mathrm{MA}_{\mathbb{C}}(\psi_\phi)(\exp(x))\,=\, \frac{1}{2^{r+n}}
\mathrm{MA}_{\mathbb{R}}(\psi_\phi)(x)\frac{1}{\mathbf J(x)}\prod_{\alpha \in \Phi_+}\alpha^2(\partial \psi_\phi(x)).
\end{equation}
In particular, by Proposition \ref{KAK int},
\begin{align}\label{volume}
{\rm Vol}(M)\,=\,\int_M \omega_\psi^n=\int_{2P_+}  \pi\,dy\,=\,V.
\end{align}
Clearly, (\ref{MA}) also holds for any K\"ahler potential in $\mathcal E^1(M,-K_M)$, which is  smooth and $K\times K$-invariant on  $G$.
For the completeness, we introduce a subspace of  $\mathcal E^1(M,-K_M)$  by
  \begin{align}\label{e-space}&\mathcal E^1_{K\times K}(M,-K_M)\notag\\
  &=\{\phi\in \mathcal E^1(M ,-K_M)|~\phi  ~{\rm is}~K\times K\text{-invariant} ~{\rm and~convex } ~{\rm on} ~G\}.
  \end{align}
  $\mathcal E^1_{K\times K}(M,-K_M)$ is locally precompact in terms of convex functions on $G$. We will also prove its completeness by using the Legendre dual in subsequent Sections 4,  6.

\subsection{Fine polytope $P$}

In this subsection, we show that the Legendre dual of Guillemin function $u_{2P}$ on $2P$ lies in $\mathcal E^1_{K\times K}(M,-K_M)$ when $P$ is fine.

Recall \eqref{polytope-eq}. For convenience, we set
$$l_A(y)\,=\,2\lambda_A-u_A(y).$$
Then
 $$2P\,=\,\cap_{A=1}^{d_0}\{l_A(y)\geq0\}.$$
Thus, $u_{2P}$ is given by (cf. \cite{Ab-JDG})
$$u_{2P}\,=\,\frac12\sum_{A=1}^{d_0}l_{A}(y)\log l_{A}(y).$$
Clearly, it is $W$-invariant, so its Legendre function $\psi_{2P}$  is also  $W$-invariant, where
\begin{align}\label{ Legendre-function}\psi_{2P}(x)\,=\,\sup_{y\in 2P} (\langle x,y\rangle- u_{2P}(y)), ~\forall~x\in \mathfrak{a}.
\end{align}
Hence, by \cite[Theorem 2]{Ab-JDG} and \cite{AL} (also see Lemma  \ref{Hessian}), \footnote{The corresponding moment map is given by ${\frac12}\nabla \psi_{2P}$, whose image is $P$.} we can extend
$$\omega_{2P}\,=\, \sqrt{-1}\partial\bar\partial \psi_{2P},~{\rm on}~\mathfrak a, $$
to a $K\times K$-invariant metric on $G$.

\begin{lemm}\label{admissible}Let $\psi_0$ be a K\"ahler potential of  admissible $K\times K$-invariant metric $\omega_0$ as in (\ref{convex-potential}).  Assume that  $P$ is  fine.  Then the K\"ahler potential $(\psi_{2P}-\psi_0)$ of $\omega_{2P}$ lies in
 $\varphi\in L^\infty(M)\cap C^\infty(M_{\text{reg}})$.  In particular,  $(\psi_{2P}-\psi_0) \in\mathcal E^1_{K\times K}(M,-K_M).$

\end{lemm}

\begin{proof}
 Fix an $m_0\in\mathbb Z_+$ such that $-m_0 K_X$ is very ample. We consider the projective embedding
$$\iota: M\to\mathbb{CP}^{N}$$
given by $|-m_0K_M|$, where $N=h^0(M,-m_0K_M)-1$. By \cite[Section 2.3]{LZ-alpha}, the pull back of the  Fubini-Study metric on $\mathbb{CP}^N$ gives a $K\times K$-invariant, Hermitian metric $h$ on $L={\mathcal O}_{\mathbb CP^N}(-1)|_M$ (also see \cite{LZ-alpha}).  Moreover, we have
\begin{eqnarray}\label{hermitian toric}
h|_{T^\mathbb C}(x)\,=\,\left(\sum_{\lambda\in mP\cap\mathfrak M}\bar n(\lambda)e^{2\lambda(x)}\right),
\end{eqnarray}
where $\bar n(\lambda)\in\mathbb Z_+$. Thus we have a K\"ahler potential  on $T^{\mathbb C}$ by
$$\psi_{FS}\,=\,\frac1{m}\log h|_{T^\mathbb C}.$$

Since $P$ is  fine, one can show directly that
$$\begin{aligned}\psi_{FS}\in\mathcal V(2P)\,=\,\{&\psi\in C^0(\mathfrak a)| ~\psi\text{ is convex, $W$-invariant}\\ &\text{and }\max_{\mathfrak a}|v_{2P}-\psi|<\infty\},\end{aligned}$$
where $v_{2P}(\cdot)$ is the support function on $\mathfrak a$ defined  by
\begin{align}\label{support-function}
v_{2P}(x)\,=\,\sup_{y\in 2P} \langle x,y\rangle.
\end{align}
Recall that the  Legendre function $u_\psi$ of $\psi$ is defined as in (\ref{ Legendre-function}) by
\begin{align}\label{ Legendre-function-1}u_\psi(y)\,=\,\sup_{x\in \mathfrak a} (\langle x,y\rangle- \psi(x)), ~y\in 2P.
\end{align}
It is known that  $\psi(x)\in\mathcal V(2P)$ if and only if $u_\psi$ is uniformly bounded on $2P$ since the Legendre function of $v_{2P}$ is zero (cf. \cite{SZ}).  Thus the Legendre function $u_{h}$ of $h|_{T^\mathbb C}(x)$ is uniformly bounded on $2P$.  It follows that
$$|u_{h}-u_{2P}|\,\le\, C.$$
Hence, we get
$$\max_\mathfrak a|\psi_{FS}-\psi_{2P}|\,<\,+\infty.$$
Consequently,
$$\max_\mathfrak a|\psi_{2P}-\psi_{0}|\,<\,+\infty.$$
By (\ref{volume}), $(\psi_{2P}-\psi_{0})$ has full MA-mass, so we have completed the proof.
\end{proof}

\section{The space $\mathcal E^1_{K\times K}(2P)$}

In this section, we describe the space $\mathcal E^1_{K\times K}(M,-K_M)$ in (\ref{e-space}) via Legendre functions as in \cite{Coman-Guedj-Sahin-Zeriahi} for   $\mathbb Q$-Fano toric varieties. Let $\psi_0$ be a K\"ahler potential of  admissible $K\times K$-invariant metric $\omega_0$ as in (\ref{convex-potential}).  Then we can normalize $\psi_0$ by (cf. \cite{LTZ}),
\begin{align}\label{normalization}
\inf_{\mathfrak a} \psi_0\,=\, \psi_0(O)\,=\,0,
\end{align}
where $O$ is the origin in  ${\mathfrak a}$.
Thus for any $\phi\in \mathcal E^1_{K\times K}(M,-K_M)$,
 $\psi_\phi=\psi_0+\phi$ can be also normalized as in (\ref{normalization}).

The following lemma is elementary.

\begin{lemm}\label{D-image}For any $K\times K$-invariant potential $\phi$ normalized as in $(\ref{normalization})$, it holds
$$\partial(\psi_\phi)\subseteq  2P,~{\rm and}~\psi_\phi\,\leq\, v_{2P},$$
where $\partial (\psi_\phi) (\cdot)$ is the normal mapping of $\psi_\phi$.

\end{lemm}

\begin{proof} We choose a sequence of decreasing and uniformly bounded $K\times K$-invariant potential $\phi_i$ normalized as in $(\ref{normalization})$ such that
$$\omega_0+\sqrt{-1}\partial\bar\partial \phi_i> 0,  ~{\rm in ~}M_{reg}$$
and
$$\phi_i\to \phi,~ {\rm as}~i\to +\infty.$$
Then
$$\sqrt{-1}\partial\bar\partial \psi_{\phi_i}> 0  \text { in }G.$$
It follows  that
$$\partial\psi_{\phi_i}\subseteq  2P.$$
This implies that $\partial\psi_\phi\subseteq  2P$. By the convexity,  we also get $\psi_\phi\leq v_{2P}$.

\end{proof}

It is easy to see that the  Legendre function $u_\phi$ of $\psi_\phi$ with  $\phi\in\mathcal E^1_{K\times K}(M.-K_M)$ satisfies
\begin{align}\label{normalization-u}
\inf_{2P}u_\phi\,=\,u_\phi(O)\,=\,0.
\end{align}
We set a class of $W$-invariant convex functions on $2P$ by
\begin{align}\mathcal E^1_{K\times K}(2P)\,=\,\{u| ~&\text{$u$ is convex, $W$-invariant on $2P$ which satisfies
(\ref{normalization-u})  and  }
\notag\\
&\int_{2P_+} u \pi\,dy<+\infty\}.\notag
\end{align}
The main goal in this section is to   prove

\begin{theo}\label{E1-legendre}
 A  K\"ahler potential $\phi\in\mathcal E^1_{K\times K}(M,-K_M)$ with
 normalized $\psi_\phi$ satisfying  \eqref{normalization} if and only if the Legendre function $u_\phi$ of $\psi_\phi$ lies in $\mathcal E^1_{K\times K}(2P)$.
As a consequence, $u_\phi$ is locally bounded in  Int$(2P)$ if $\phi\in\mathcal E^1_{K\times K}(M,-K_M)$.
\end{theo}

As in \cite{Coman-Guedj-Sahin-Zeriahi}, we need to establish a comparison principle for  the complex  Monge-Amp\`ere measure
 in  $\mathcal E^1_{K\times K}(M,-K_M)$.
For our purpose,  we  will introduce a weighted Monge-Amp\`ere measure in the following.

\subsection{Weighted Monge-Amp\'ere measure}

\begin{defi}\label{weighted-ma-measure}
Let $\Omega\subseteq \mathfrak a$ be a $W$-invariant domain and $\psi$
 any $W$-invariant convex function on $\Omega$. Define a weighted Monge-Amp\'ere measure on $\Omega$  by
$$\int_{\Omega'}\text{MA}_{\mathbb R;\pi}(\psi)dx\,=\,\int_{\partial \psi(\Omega')}\pi\,dy,~\forall\text{ $W$-invariant }\Omega'\Subset\Omega,$$
where $\partial \psi(\cdot)$ is the normal mapping of $\psi$.
\end{defi}

\begin{rem}\label{weak-conv}
Let $\psi_k$ be a sequence of convex functions which converge  locally uniformly to $\psi$ on $\Omega$, then $\text{MA}_{\mathbb R;\pi}(\psi_k)$ converge   to $\text{MA}_{\mathbb R;\pi}(\psi)$ (cf. \cite[Section 15]{Aleksandrov-1940}).
This follows from the fact:
$$\liminf_{k\to+\infty}\partial \psi_k(U)\supseteq\partial \psi(U),~\forall ~\text{$W$-invariant open subset }U\subseteq\Omega.$$
\end{rem}

\begin{lemm}\label{approximate-lemma}
Let $\omega_\phi\,=\,\sqrt{-1}\partial\bar\partial \psi_\phi$ with $\phi\in\mathcal E^1_{K\times K}(M,-K_M)$. Then for any $K\times K$-invariant continuous uniformly bounded function $f$  on $G$, it holds
\begin{align}\label{weight-ma-measure}\int_{M}f\omega_\phi^n\,=\,\int_{\mathfrak a_+}f\text{MA}_{\mathbb R;\pi}(\psi_\phi)dx.
\end{align}
\end{lemm}

\begin{proof} First we assume that $f$ is a  $K\times K$-invariant continuous function with compact support on $ {\mathfrak a}$. We take
  a sequence of  smooth $W$-invariant convex functions $\psi_k\searrow\psi$ and   let
$\omega_k=\sqrt{-1}\partial\bar\partial\psi_k$.  Then for any $W$-invariant $\Omega'\Subset  {\mathfrak a}$,
 it holds
$$\int_{\Omega'}\text{MA}_{\mathbb R;\pi}(\psi_k)dx\,:=\,\int_{\Omega'}\det(\nabla^2\psi_k)\pi(\nabla \psi_k)\,dy.$$
By the standard $KAK$-integration formula, it follows that
$$\int_{M}f\omega_k^n\,=\,\int_{\mathfrak a_+}f\det(\nabla^2\psi_k)\pi(\nabla\psi_k)dx\,=\,\int_{\mathfrak a_+}f\text{MA}_{\mathbb R;\pi}(\psi_k)dx.$$
Since
$$\int_{M}f\omega_k^n\,\to\,\int_{M}f\omega^n,$$
it follows from  Remark \ref{weak-conv} that (\ref{weight-ma-measure}) is true.

 Next we choose a sequence of  exhausting $W$-invariant convex domains $\Omega_k$ in ${\mathfrak a}$ and a sequence of $W$-invariant convex functions  with  the support   on $\Omega_{k+1}$ such that
 $f_k=f|_{\Omega_k}$.  Since $\omega^n$ has  full MA-mass,  we get
 \begin{align}\int_{M}f_k\omega^n&=\,\lim_k\int_{M}f_k\omega^n\notag\\
& =\,\lim_k\int_{\mathfrak a_+}f_k\text{MA}_{\mathbb R;\pi}(\psi_\phi)dx\notag\\
&=\,\lim_k\int_{\mathfrak a_+}f\text{MA}_{\mathbb R;\pi}(\psi_\phi)dx.\notag
\end{align}

\end{proof}

\subsection{Comparison principles}
We establish the following  comparison principle for the weighted Monge-Amp\'ere measure $\text{MA}_{\mathbb R;\pi}(\psi)$.

\begin{prop}\label{compare}
Let $\Omega\subseteq\mathfrak a$ be a $W$-invariant domain and $\varphi,\psi$ be two convex functions on $\Omega$ such that
\begin{align}\label{010101}
\varphi\geq\psi\text{ and }(\varphi-\psi)|_{\partial\Omega}\,=\,0.
\end{align}
Then
\begin{align}\label{010102}
\int_{\Omega}\text{MA}_{\mathbb R;\pi}(\varphi)dx\,\leq\,\int_{\Omega}\text{MA}_{\mathbb R;\pi}(\psi)dx.
\end{align}
\end{prop}

\begin{proof}
It is sufficient to prove \eqref{010102} when $\varphi$ and $\psi$ are smooth, since we can approximate general $\varphi$ and $\psi$ by smooth $W$-invariant convex functions by Lemma \ref{approximate-lemma}.  Let
$$\varphi_t\,=\,t\varphi+(1-t)\psi.$$
Then
$$\text{MA}_{\mathbb R;\pi}(\varphi_t)=\det(\nabla^2\varphi_t)\prod_{\alpha\in\Phi_+}\alpha^2(\nabla\varphi_t)$$
and
\begin{align}\label{010103}
&\frac{d}{dt}\int_{\Omega}\det(\nabla^2\varphi_t)\prod_{\alpha\in\Phi_+}\alpha^2(\nabla\varphi_t)dx\notag\\
=&\int_{\Omega}(\nabla^2\varphi_t)^{-1,ij}\nabla^2\dot\varphi_{t,ij}\det(\nabla^2\varphi_t)\prod_{\alpha\in\Phi_+}\alpha^2(\nabla\varphi_t)dx\notag\\
&+\int_{\Omega}\left(\sum_{\alpha\in\Phi_+}\frac{2\alpha(\nabla\dot\varphi_t)}{\alpha(\nabla\varphi_t)}\right)\det(\nabla^2\varphi_t)\prod_{\alpha\in\Phi_+}\alpha^2(\nabla\varphi_t)dx.
\end{align}
Using the fact that
$$\left(\det(\nabla^2\varphi_t)(\nabla^2\varphi_t)^{-1,ij}\right)_{,j}=0$$
and integration by parts, we have
\begin{align}\label{010104}
&\int_{\Omega}(\nabla^2\varphi_t)^{-1,ij}\nabla^2\dot\varphi_{t,ij}\det(\nabla^2\varphi_t)\prod_{\alpha\in\Phi_+}\alpha^2(\nabla\varphi_t)dx\notag\\
=&\int_{\partial\Omega}(\nabla^2\varphi_t)^{-1,ij}\nabla\dot\varphi_{t,i}\nu_j\det(\nabla^2\varphi_t)\prod_{\alpha\in\Phi_+}\alpha^2(\nabla\varphi_t)d\sigma\notag\\
&-\int_{\partial\Omega}[(\nabla^2\varphi_t)^{-1,ij}\det(\nabla^2\varphi_t)\prod_{\alpha\in\Phi_+}\alpha^2(\nabla\varphi_t)]_{,j}\nu_i\dot\varphi_t d\sigma\notag\\
&+\int_{\Omega}(\nabla^2\varphi_t)^{-1,ij}\dot\varphi_{t}\det(\nabla^2\varphi_t)\left(\prod_{\alpha\in\Phi_+}\alpha^2(\nabla\varphi_t)\right)_{,ij}dx.
\end{align}
Also
\begin{align}\label{010105}
&\int_{\Omega}\left(\sum_{\alpha\in\Phi_+}\frac{2\alpha(\nabla\dot\varphi_t)}{\alpha(\nabla\varphi_t)}\right)\det(\nabla^2\varphi_t)\prod_{\alpha\in\Phi_+}\alpha^2(\nabla\varphi_t)dx\notag\\
=&2\int_{\partial\Omega}\sum_{\alpha\in\Phi_+}\frac{\alpha^i\nu_i}{\alpha(\nabla\varphi_t)}\det(\nabla^2\varphi_t)\prod_{\alpha\in\Phi_+}\alpha^2(\nabla\varphi_t)\dot\varphi_t d\sigma\notag\\
=&-2\int_{\Omega}\left(\det(\nabla^2\varphi)\prod_{\alpha\in\Phi_+}\alpha^2(\nabla\varphi_t)\sum_{\alpha\in\Phi_+}\frac{\alpha^i}{\alpha(\nabla\varphi)}\right)_{,i}\dot\varphi_t dx.
\end{align}
Note that
\begin{align}\label{010106}
&\frac{\left(\prod_{\alpha\in\Phi_+}\alpha^2(\nabla\varphi_t)\right)_{,ij}}{\prod_{\alpha\in\Phi_+}\alpha^2(\nabla\varphi_t)}\notag\\
=&-2\sum_{\alpha\in\Phi_+}\frac{\alpha^k\alpha^l\varphi_{t,ik}\varphi_{t,jl}}{\alpha^2(\nabla\varphi_t)}+2\sum_{\alpha\in\Phi_+}\frac{\alpha^k\varphi_{t,ijk}}{\alpha(\nabla\varphi_t)}+4\sum_{\alpha,\beta\in\Phi_+}\frac{\alpha^k\beta^l\varphi_{t,ik}\varphi_{t,jl}}{\alpha(\nabla\varphi_t)\beta(\nabla\varphi_t)}\notag\\
=&2\frac{\left(\det(\nabla^2\varphi)\prod_{\alpha\in\Phi_+}\alpha^2(\nabla\varphi_t)\sum_{\alpha\in\Phi_+}\frac{\alpha^i}{\alpha(\nabla\varphi)}\right)_{,i}}{\det(\nabla^2\varphi)\prod_{\alpha\in\Phi_+}\alpha^2(\nabla\varphi_t)}.
\end{align}
Plugging \eqref{010104}-\eqref{010106} into \eqref{010103} and using the boundary condition \eqref{010101}, we have
\begin{align*}
&\frac{d}{dt}\int_{\Omega}\det(\nabla^2\varphi_t)\prod_{\alpha\in\Phi_+}\alpha^2(\nabla\varphi_t)dx\\
=&\int_{\partial\Omega}\sum_{\alpha\in\Phi_+}\nabla\dot\varphi_{t,i}\nu_j\det(\nabla^2\varphi_t)\prod_{\alpha\in\Phi_+}\alpha^2(\nabla\varphi_t)\sigma\\
\leq&0.
\end{align*}
Hence we get \eqref{010102}.
\end{proof}

By the above proposition,   we get the following analogue of \cite[Lemma 2.3]{Coman-Guedj-Sahin-Zeriahi}.

\begin{lemm}\label{compare-global}
Let $\varphi,\psi$ be two $W$-invariant convex functions on $\mathfrak a$ so that
$$\varphi\,\geq\,\psi$$ and $$\lim_{|x|\to+\infty}\varphi(x)\,=\,+\infty.$$
Then
$$\int_{\mathfrak a_+}\text{MA}_{\mathbb R;\pi}(\varphi)dx\,\geq\,\int_{\mathfrak a_+}\text{MA}_{\mathbb R;\pi}(\psi)dx.$$
\end{lemm}

Combining Lemma \ref{compare-global} and the argument of \cite[Lemma 2.7]{Coman-Guedj-Sahin-Zeriahi}, we prove

\begin{lemm}\label{energy-preserve}
Let $\psi$ be a $W$-invariant convex function on $\mathfrak a$ and $u$ its Legendre function. Suppose that for some constant $C$,
\begin{align}\label{up-bound}
\psi\,\leq\, v_{2P}+C,
\end{align}
where $v_{2P}$ is the support function of $2P$.  Then
\begin{align}\label{full-mass}\int_{\mathfrak a_+}\text{MA}_{\mathbb R;\pi}(\psi)dx\,=\,\int_{2P}\pi\,dy,
\end{align}
if $u<+\infty$ everywhere in the interior of $2P$.
\end{lemm}

The inverse of Lemma \ref{energy-preserve} is also true as an analogue of \cite[Theorem 3.6]{Coman-Guedj-Sahin-Zeriahi}. In fact, we have

\begin{prop}\label{E-space}
Let   $\phi$ be a   $K\times K$-invariant potential.  Then $\psi_\phi$ satisfies \eqref{full-mass} if and only if
$u_\phi$ is finite everywhere in Int$(2P)$.
\end{prop}

By Proposition \ref{E-space}, we will follow the arguments in  \cite[Proposition 3.9]{Coman-Guedj-Sahin-Zeriahi} to prove Theorem \ref{E1-legendre}.

\subsection{Proof of Theorem \ref{E1-legendre}}

It is easy to see that \eqref{normalization} is equivalent to \eqref{normalization-u}.  Thus, to prove Theorem \ref{E1-legendre},  we only need to show that
$$\phi\in\mathcal E^1_{K\times K}(M,-K_M)\Longleftrightarrow \int_{2P_+}|u_\phi|\pi\,dy<+\infty.$$

The following lemma can be found in  \cite[Lemma 2.7]{Berman-Berndtsson} (proved in \cite[Appendix]{Berman-Berndtsson}).

\begin{lemm}\label{derivative-dual} Let $\psi$ be a convex function  on  ${\mathfrak a}$ and $u_\psi$ its Legendre dual on $P$.
\begin{itemize}
\item[(1)]  $u_\psi$ is differentiable  at $p$ if and only if $u_\psi$ is attained at a unique point
 $x_p\in {\mathfrak a}$ and $x_p=\nabla u_\psi(p)$;

\item[(2)]  Suppose that  $(\psi-\psi_0)\in \mathcal E^1_{K\times K}(M,-K_M)$.  Let $p\in P$ at  which       $u_\psi$ is  differentiable.   Then for any  continuous uniformly bounded function $v$
on ${\mathfrak a}$,  it holds
\begin{align}\label{dual-derivative}
\left.\frac{d}{dt}\right|_{t=0} u_{\psi+tv}(p)=-v(\nabla u_\psi(p)),
\end{align}
where  $u_{\psi+tv}$ is the Legendre function of  $\psi+tv$  as in (\ref{ Legendre-function-1}) which is well-defined since $ v$ is continuous  and uniformly bounded on  ${\mathfrak a}$.
\end{itemize}
\end{lemm}

\begin{rem}\label{integral-formula} By Lemma \ref{approximate-lemma} and  Part (1) in Lemma \ref{derivative-dual}, we can prove the following:
Let $\phi\in\mathcal E^1_{K\times K}(M,-K_M)$, then for any $K\times K$-invariant continuous uniformly bounded function $f$  on $G$,  it holds
\begin{align}\label{transfer-integral}\int_{M}f\omega_\phi^n\,=\,\int_{2P}f (\partial u_\phi)\pi dy.
\end{align}

\end{rem}

\begin{proof}[Proof of Theorem \ref{E1-legendre}]
\textbf{Necessary part.}
First we show that  $\phi$ has  full MA-mass  by  Proposition \ref{E-space}. In fact,
by a result in   \cite[Lemma 4.5]{LZZ},  we see that for any $W$-invariant convex polytope $2P'\subseteq 2P$, there is a constant $C=C(P')$ such that for any $W$-invariant convex $u_\phi\geq0,$
$$\int_{2P'} u_\phi  \,dy\,\le\,  C\int_{2P} u_\phi \pi \,dy  \, <\, +\infty.$$
This implies  that $u_\phi$ is  finite everywhere in Int$(2P)$ by the convexity of  $u_\phi$.
Thus we get what we want from Proposition \ref{E-space}.

Next we prove that  $\phi$ is $L^1$-integrate associated to the MA-measure $\omega_{\phi}^n$.
 Let $\psi_1=\psi_0+\phi$ ($\phi$ may be different to a constant).  We
define a distance between $\psi_0$ and $\psi_1$ for $p\geq1$,
$$d_p(\psi_0,\psi_1)\,=\,\inf_{\phi_t}\int_0^1\left(\int_M|\dot\phi_t|^p\omega^n_{\phi_t}\right)^{\frac1p}dt,$$
where $\phi_t\in\mathcal E^1(M,-K_M)$   $(t\in[0,1])$ runs over all curves  joining $0$ and $\phi$ with $\omega_{\phi_t}\ge 0$.     Choose  a special path ${\phi_t}$  such that the corresponding Legendre functions of $\psi_t=\psi_0+\phi_t$ are given by
\begin{align}\label{linear-path}
u_t\,=\,tu_1+(1-t)u_0,
\end{align}
where $u_1$ and  $u_0$ are the Legendre functions of $\psi_1$ and $\psi_0$, respectively.
Note that  by Lemma \ref{derivative-dual},
$$\dot\psi_t=-\dot u_t=u_0-u_1,~{\rm almost ~everywhere}.$$
 Then by Lemma \ref{approximate-lemma} (or Remark \ref{integral-formula}), we get
\begin{align}
d_p(\psi_0,\psi_1)&\leq\int_0^1\left(\int_{2P_+}|\dot u_{t}|^p\pi\,dy\right)^{\frac1p}dt\notag\\
&\leq C(p)\left(\int_{2P_+}|u_1|^p\pi\,dy\right)^{\frac1p}+C'(p.\psi_0).
\end{align}
On the other hand, by a result of Darvas \cite{DR},
there are  uniform constant $C_0$ and $C_1$  such that for any  K\"ahler potential $\phi$ with full MA-measure
it holds,
$$-\int_M \phi\omega^n_\phi\,\le\,  C_0 d_1(\psi_0,\psi_1)+C_1. $$
Thus we obtain
$$-\int_M \phi\omega^n_\phi\,\le\,  C.$$
Hence, $\phi\in\mathcal E^1_{K\times K}(M,-K_M)$.

\textbf{Sufficient part.}   Assume that $\phi\in\mathcal E^1_{K\times K}(M,-K_M)$.  We first deal with the case of  $\phi\in L^\infty(M)\cap C^\infty(G)$.  Then
\begin{align}\label{020202}
v_{2P}-C\,\leq\,\psi_\phi\,\leq\, v_{2P}\,\le\, \psi_0+C,
\end{align}
and
 $$\nabla \psi_\phi:\mathfrak a\to 2P$$
is a bijection.
Thus
\begin{align*}
-\phi&=(\psi_0-\psi_\phi)(\nabla u_\phi)\notag\\
&\geq v_{2P}(\nabla u_\phi)-\psi_\phi(\nabla u_\phi)-C_2\geq -C_2.\notag
\end{align*}
Moreover,
\begin{align*}
&(\psi_0-\psi_\phi)(\nabla u_\phi)\notag\\
\ge&  v_{2P}(\nabla u_\phi)-   \psi_\phi(\nabla u_\phi) -C\notag\\
=&\sup_{y'\in 2P} \langle  \nabla u_\phi,   y' \rangle-   \psi(\nabla u_\phi) -C\notag\\
\ge& \langle  \nabla u_\phi,   y\rangle-   \psi(\nabla u_\phi) -C\notag\\
=& u_\phi(y)-C.\notag
\end{align*}
Hence,
\begin{align}\label{020203}
\int_{2P_+} u_\phi \pi\,dy&\leq\int_{2P_+} (\psi_0-\psi_\phi)(\nabla u_\phi) \pi\,dy +C\notag\\
&= \int_M|\phi|\omega_\phi^n +C<+\infty.
\end{align}

For an arbitrary $\phi\in\mathcal E^1_{K\times K}(M,-K_M)$,   we choose  a sequence of smooth $K\times K$-invariant functions $\{\phi_j\}$ decreasing to $\phi$
such that $\phi_j\in C^\infty(G)$ and
$$\sqrt{-1}\partial\bar\partial (\psi_0 +\phi_j)> 0,~{\rm in ~} G. $$
Then as in  (\ref{020203}),  we have
\begin{align}
\int_{2P_+} u_j \pi\,dy&\leq\int_{2P_+} (\psi_0-\psi_j)(\nabla u_j)\pi\,dy\notag\\
&=\int_M|\phi_j|\omega_j^n +C,\notag
\end{align}
where $u_j$ is the Legendre functions of   $\psi_j=\psi_0 +\phi_j.$
 Note that
$$\int_M|\phi_j|\omega_j^n\to\int_M|\phi|\omega_\phi^n$$
and $u_{j}\nearrow u_\phi$.
Thus   by taking the above limit as $j\to+\infty$,  we get \eqref{020203} for $\phi$.
In particular,
$$ \int_{2P_+} u_\phi\pi\,dy <+\infty.$$

\end{proof}

\section{Computation of Ricci potential}

In this section, we assume that the moment polytope $P$  of $Z$ is fine.
  Then by Lemma \ref{admissible},
 $(\psi_{2P}-\psi_0)\in\mathcal E^1_{K\times K}(M, -K_M)$ is a smooth $K\times K$-invariant    K\"ahler potential on $G$.  It follows that
\begin{equation}\label{h0}
-\log\det(\partial\bar{\partial}\psi_{2P})-\psi_{2P}\,=\,h_0
\end{equation}
gives a Ricci potential $h_0$ of $\omega_{2P}$, which is  smooth  and  $K\times K$-invariant  on $G$.

The following proposition gives an upper bound on $h_0$.

\begin{prop}\label{ricci-potential}
 The Ricci potential $h_0$ of $\omega_{2P}$ is uniformly bounded from above on $G$.  In particular, $e^{h_0}$ is uniformly  bounded on $G$.
\end{prop}

\begin{proof} As in \cite[Sections 3.2 and 4.3]{LZZ},
 the  proof is  based on a direct computation of  asymptotic behavior  of $h_0$ near every point of $\partial(2P_+)$.
 Recall that
$$\mathbf J(x)\,=\,\prod_{\alpha\in\Phi_+}\sinh^2\alpha(x)\text{ and }\pi(y)=\prod_{\alpha\in\Phi_+}\alpha^2(y).$$
Since the Ricci potential of $h_0$ is also $K\times K$-invariant, by \eqref{h0} and (\ref{MA}),
\begin{equation}\label{5115}
\begin{aligned}
h_0&=-\log\det(\psi_{{2P},ij})-\psi_{2P}+\log \mathbf J(x)-\log\prod_{\alpha\in\Phi_+}\alpha^2(\nabla\psi_{2P})\\
&=\log\det(u_{{2P},ij})-y_iu_{{2P},i}+u_{2P}+\log \mathbf J(\nabla u_{2P})-\log\pi(y).
\end{aligned}
\end{equation}
Note that
\begin{align*}
u_{{2P},i}&=\frac12\sum_{A=1}^{d_0}(-u_A^i)(1+\log l_A),\\
u_{{2P},ij}&=\frac12\sum_{A=1}^{d_0}\frac{u_{{2P}}^i u_{{2P}}^j}{l_A}
\end{align*}
and
\begin{align*}
\log \mathbf J(t)&=2\sum_{\alpha\in\Phi_+}\log\sinh (t).
\end{align*}
Thus we have
\begin{align}
h_0=&-\sum_{A=1}^{d_0}\log l_A+\frac12\sum_{A=1}^{d_0}(u_A^iy_i)\log l_A\notag\\
&+2\sum_{\alpha\in\Phi_+}\log\sinh(-\frac12\sum_{A=1}^{d_0}\alpha(u_A)\log l_A)-2\sum_{\alpha\in\Phi_+}\log\alpha(y)+O(1).\label{h0-sum}
\end{align}

By  \eqref{h0-sum},  $h_0$ is locally bounded in the interior of $2P_+$.  Thus we need  to prove  that $h_0$ is bounded  from above  near each $y_0\in\partial(2P_+)$.
There will be  three cases as follows.

\emph{Case-1.} $y_0\in\partial(2P_+)$ and is away from any Weyl wall.
\begin{figure}[h]
\begin{center}
\includegraphics[width=1.5in]{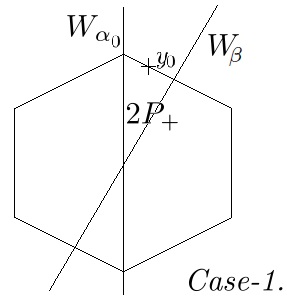}\\
\end{center}
\end{figure}
Note that
\begin{align}
\log\sinh(t)&=\left\{\begin{aligned}&t+O(1),~t\to +\infty,\\&\log t+O(1),~t\to 0^+.\end{aligned}\right.\label{logsinh}
\end{align}
Then we get as  $y\to y_0$,
\begin{align*}
\sum_{\alpha\in\Phi_+}\log\sinh(-\frac12\sum_{A=1}^{d_0}\alpha(u_A)\log l_A)=
-\sum_A\rho(u_A)\log l_A+O(1).
\end{align*}
 By \eqref{h0-sum},  it follows that
\begin{equation}\label{5116}
h_0=-\sum_{\{A|l_A(y_0)=0\}}\left(1-\frac12y_iu_{A}^i+2\rho_iu_{A}^i\right)\log l_{A}(y)+O(1).\nonumber
\end{equation}
However, by  Lemma \ref{polytope-coefficient}, we have
\begin{equation}\label{5116}
h_0=-\frac12\sum_{\{A|l_A(y_0)=0\}}l_{A}(y)\log l_{A}(y)+O(1).\nonumber
\end{equation}
Hence $h_0$ is bounded near $y_0$.

\emph{Case-2.} $y_0$ lies on some Weyl walls but away from any facet  of $2P$.
\begin{figure}[h]
\begin{center}
  \includegraphics[width=1.5in]{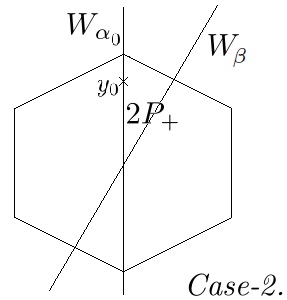}\\
\end{center}
\end{figure}
In this case it is direct to see that $h_0$ is bounded near $y_0$ since
$$\log\det(u_{{2P},ij}),~y_iu_{{2P},i}, ~\frac{\mathbf J(\nabla u_{2P})}{\pi(y)}$$
are all bounded.

\emph{Case-3.} $y_0$ lie on the intersection of $\partial (2P)$ with some Weyl walls. In this case, by \eqref{polytope-coefficient}, we rewrite \eqref{h0-sum} as
\begin{align}
h_0=&2\sum_{A=1}^{d_0}\rho_A(u_A)\log l_A+2\sum_{\alpha\in\Phi_+}\log\sinh(-\frac12\sum_{A=1}^{d_0}\alpha(u_A)\log l_A)\notag\\
&-2\sum_{\alpha\in\Phi_+}\log\alpha(y)+O(1)\notag\\
=&\sum_{\alpha\in\Phi_+}\left[\sum_{A=1}^{d_0}|\alpha(u_A)|\log l_A+2\log\sinh(-\frac12\sum_{A=1}^{d_0}\alpha(u_A)\log l_A)\right.\notag\\
&\left.-2\log\alpha(y)\right]+O(1),~y\to y_0.\notag
\end{align}
Here we used a fact that
\begin{align*}
2\rho_A(u_A)=\sum_{\alpha\in\Phi_+}|\alpha(u_A)|.
\end{align*}

Set
$$I_\alpha(y)=\sum_{A=1}^{d_0}|\alpha(u_A)|\log l_A+2\log\sinh(-\frac12\sum_{A=1}^{d_0}\alpha(u_A)\log l_A)-2\log\alpha(y)$$
for each $\alpha\in\Phi_+$. Then
\begin{align}\label{h0-sum-1}
h_0(y)=\sum_{\alpha\in\Phi_+}I_\alpha(y)+O(1),~y\to y_0.
\end{align}
Note that each $I_\alpha(y)$  involves only one root $\alpha$.  Thus, without  loss of generality,  we may assume that  $y_0$ lies on only one Weyl wall.

Assume that $y_0\in \partial (2P)\cap W_{\alpha_0}$ for some simple Weyl wall $W_{\alpha_0}$,  $\alpha_0\in\Phi_+$ and it  is away from other Weyl walls.  Now we estimate each   $I_\alpha(y)$ in  (\ref{h0-sum-1}). When $\beta\not=\alpha_0$,  it is easy to see that
$$\beta(y)\to c_\beta>0,~ {\rm as}~y\to y_0.$$
Then, by \eqref{logsinh},  we have
\begin{align*}
\log\sinh(-\frac12\sum_{A=1}^{d_0}\beta(u_A)\log l_A)&=-\frac12\sum_{\{A|l_A(y_0)=0\}}\beta(u_A)\log l_A+O(1),
\forall\beta\not=\alpha_0.
\end{align*}
Note that $y_0\in\{\beta(y)>0\}$. Thus any facet $F_A$ passing through  $y_0$ lies in  $\{\beta(y)>0\}$ or is orthogonal to $W_\beta$.  Since $2P$ is convex and $s_\beta$-invariant, where $s_\beta$ is  the reflection with respect to $W_\beta$,  these facets must satisfy
 $$\beta(u_A)\ge 0.$$
Hence,    for any $\beta\not=\alpha_0,$ we get
\begin{align}\label{beta-term}
I_\beta(y)=&\sum_{A=1}^{d_0}|\beta(u_A)|\log l_A-2\sum_{A=1}^{d_0}|\beta(u_A)|\log l_A-2\log\beta(y)\notag\\
=&O(1),~{\rm as}~y\to y_0.
\end{align}

 It remains to estimate  the second term in $I_{\alpha_0}(y)$,
 \begin{align}\label{bad-term}
 \log\sinh(-\frac12\sum_A\alpha_0(u_A)\log l_A).
 \end{align}
We first consider a simple case that  $y_0$ lies on the intersection of $W_{\alpha_0}$ with at most two facets of $2P$. Then  there will be  two subcases:  $y_0\in W_{\alpha_0} \cap F_1$ where $F_1$ is orthogonal to $W_{\alpha_0}$, or $y_0\in  W_{\alpha_0}\cap F_1\cap F_2$, where $F_1, F_2$ are two facets of $P$.

\emph{Case-3.1.} $y_0\in W_{\alpha_0} \cap F_1$  is away from other facet of $2P$.  Then  $F_1$ is orthogonal to $W_{\alpha_0}$.
\begin{figure}[h]
\begin{center}
  \includegraphics[width=1.5in]{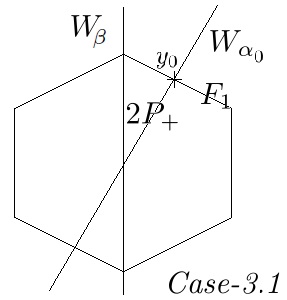}\\
\end{center}
\end{figure}
It follows that  $l_A(y_0)\not=0$ for any $A\not=1$.  Thus
\begin{align}\label{small-inner} \langle\alpha_0,y\rangle=o(l_A(y)),y\to y_0,A\not=1.
\end{align}

Let $\{F_1,...,F_{d_1}\}$ be all facets of $P$ such that $\alpha_0(u_{A})\geq0,A=1,...,d_1$. Let $s_{\alpha_0}$ be the reflection with respect to $W_{\alpha_0}$.  Then by $s_{\alpha_0}$-invariance of $P$, for each $A'\not\in\{1,...,d_1\}$ there is some $A\in\{1,...,d_1\}$ such that
$$l_{A'}=l_A+2\frac{\alpha_0(u_A)\langle\alpha_0,y\rangle}{|\alpha_0|^2}.$$
It follows that
\begin{align}
\alpha_0(\nabla u_{2P})&=-\frac12\sum_{A=1}^{d_0}\alpha_0(u_A)\log l_A\notag\\
&=\frac12\sum_{A=2}^{d_1}\alpha_0(u_A)\log\left(1+2\frac{\alpha_0(u_A)\langle\alpha_0,y\rangle}
{|\alpha_0|^2l_A(y)}\right).\notag
\end{align}
Thus, by (\ref{small-inner}) and  the fact that $\alpha_0(u_1)=0$, we obtain
\begin{align*}
&\log\sinh(-\frac12\sum_{A=1}^{d_0}\alpha_0(u_A)\log l_A)\\
&=\log\sinh\sum_{A=2}^{d_1}\alpha_0(u_A)\log\left(1+2\frac{\alpha_0(u_A)\langle\alpha_0,y\rangle}{|\alpha_0|^2l_A(y)}\right)\\
&=\log \langle\alpha_0,y\rangle+O(1).
\end{align*}
Hence
\begin{align*}
&I_{\alpha_0}(y)=O(1),~\text{as }y\to y_0.
\end{align*}
Together with \eqref{beta-term}, we see that
$h_0$ is  bounded near $y_0$.

\emph{Case-3.2.} $y_0\in W_{\alpha_0} \cap F_1\cap F_2$  and is away from other facets  of $2P$.
\begin{figure}[h]
\begin{center}
  \includegraphics[width=1.5in]{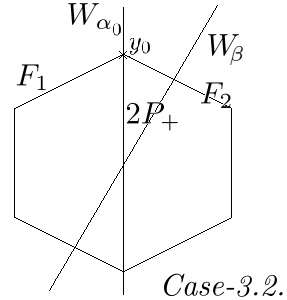}\\
\end{center}
\end{figure}
 By the $W$-invariance of $2P$, it must hold $F_1=s_{\alpha_0}(F_2)$. We may assume that $F_2\subseteq\overline{\mathfrak a_+}$ and then
$$l_1=l_2+\frac{2\alpha_0(u_2)\langle\alpha_0,y\rangle}{|\alpha_0|^2}.$$

As $y\to y_0$ we have
\begin{align*}
&\alpha_0(y),l_1(y),l_2(y)\to0,\\
&l_A(y)\not\to0,~\forall A\not=1,2.
\end{align*}
It follows that
\begin{align}
\sum_{A=1}^{d_0}|\alpha_0(u_A)|\log l_A=&\alpha_0(u_2)(\log l_1+\log l_2)+O(1).\label{other-terms-1}
\end{align}
Then  the second term in $I_{\alpha_0}(y)$ becomes
\begin{align*}
&\log\sinh(-\frac12\sum_{A=1}^{d_0}\alpha_0(u_A)\log l_A)\\
&=\log\sinh\frac12\left[\alpha_0(u_2)\log\left(1+2\frac{\alpha_0(u_2)\langle\alpha_0,y\rangle}{|\alpha_0|^2l_2(y)}\right)\right.\\
&\left.+\sum_{A\not=2,\alpha_0(u_A)>0}^{d_1}\alpha_0(u_A)\log\left(1+2\frac{\alpha_0(u_A)\langle\alpha_0,y\rangle}{|\alpha_0|^2l_A(y)}\right)\right].\\
\end{align*}
We will settle it down according  to the different rate of $\frac{\alpha_0(y)}{l_2(y)}$ below.

\emph{Case-3.2.1.} $\alpha_0(y)=o(l_2(y))$. Then
\begin{align}\label{log-sinh-alpha-321}
\log\sinh\alpha_0(\nabla u_{2P})=\log\alpha_0(y)-\log l_2(y).
\end{align}
Note that $s_{\alpha_0}(u_1)=u_2\in\overline{\mathfrak a_+}$, we have
$$\sum_{A=1,2}|\alpha_0(u_A)|\log l_A=\alpha_0(u_2)(\log l_1+\log l_2).$$
Using the above relation, \eqref{other-terms-1} and \eqref{log-sinh-alpha-321}, we get
\begin{align}\label{5116+}
I_{\alpha_0}(y)=&\alpha_0(u_2)\log l_1+(\alpha_0(u_2)-2)\log l_2+O(1)\notag\\
=&2(\alpha_0(u_2)-1)\log l_2+O(1).
\end{align}
Here we used $l_1=l_2(1+o(1))$ in the last equality.

Note that by our assumption
 $\alpha_0(u_2)>0$. Then
  $$\alpha_0(u_2)\geq1, $$
 since $\alpha_0(u_2)\in\mathbb Z$.
Hence, as  $l_1(y), l_2(y) \to0^+$, by \eqref{h0-sum-1}, \eqref{beta-term} and \eqref{5116+},  we see that $h_0$ is bounded from above in this case.

\emph{Case-3.2.2.} $c\leq\frac{\alpha_0(y)}{l_2(y)}\leq C$ for some constants $C,c>0$. Then
\begin{align*}
\log\alpha_0(y)=\log l_2+O(1),~\log\sinh\alpha_0(\nabla u_{2P})=O(1)
\end{align*}
and the right hand side of \eqref{h0-sum-1}  becomes
\begin{align}\label{5117}
&\alpha_0(u_2)(\log l_1+
\log l_2)-2\log \alpha_0(y)+O(1)\notag\\
=&2(\alpha_0(u_2)-1)\log l_2+O(1).
\end{align}
Again $h_0$ is  also  bounded from above.

\emph{Case-3.2.3.} $\frac{\alpha_0(y)}{l_2(y)}\to+\infty$. Then
\begin{align*}
\log\sinh\alpha_0(\nabla u_{2P})=&\frac12\alpha_0(u_2)(\log\alpha_0(y)-\log{l_2(y)}),\\ l_1(y)=&\alpha_0(y)(1+o(1))
\end{align*}
and the right hand side of \eqref{h0-sum-1} becomes
\begin{align}\label{5118}
&\alpha_0(u_2)(\log l_1+\log l_2)+\alpha_0(u_2)(\log\alpha_0(y)-\log l_2(y))\notag\\&-2\log\alpha_0(y)+O(1)\notag\\
=&\alpha_0(u_2)\log l_1+[\alpha_0(u_2)-2]\log \alpha_0(y)+O(1)\notag\\
=&2(\alpha_0(u_2)-1)\log \alpha_0(y)+O(1).
\end{align}
Hence
$h_0$ is  bounded from above as in  \emph{Case-3.2.1}.

Next we consider the case that  there are facets  $F_1...,F_s~(s \ge 3)$ such that
$$y_0\in W_{\alpha_0}\cap F_1\cap...\cap F_s $$
and it
is away from any other facet of $2P$.  We only need to control the term  (\ref{bad-term}) as above. If  $F_1,...,F_s$ are all orthogonal to $W_{\alpha_0}$ as in  \emph{Case-3.1}, we see that $h_0(y)$ is uniformly  bounded. Otherwise, for any $y$ nearby $y_0$   there is a  facet  $F=F_{i'}$  for some $i'\in \{1,...,s\}$ such that
 $$l_{i'}(y)=\min\{l_i(y)|~i=1,...,s\text{ such that }\alpha_0(u_i)\not=0\}.$$
As $y\to y_0$, up to passing to a subsequence, we can fix this $i'$. Clearly, $y_0\in W_{\alpha_0}\cap F_1\cap  F_2$ as in   \emph{Case-3.2}, where   $F_2=F\subseteq\overline{\mathfrak a_+}$ and $F_1=s_{\alpha_0}(F)$ for the reflection $s_{\alpha_0}$.   Hence by following   the argument in  \emph{Case-3.2},  we can  also prove   that $h_0(y)$ is uniformly  bounded from above.
Therefore,  the proposition  is true in \emph{Case-3}.
The proof of our proposition is completed.
\end{proof}

\begin{rem}\label{h-remark} We note that $h_0$ is always uniformly bounded  in {Case-1},  {Case-2} and {Case-3.1}.  Furthermore, if rank$(G)=2$,  there are at most  two facets $F_1,F_2$ intersecting at a same point $y_0$ of $W_{\alpha_0}$ as in {Cases-3.2.1-3.2.3}, thus, by the asymptotic expressions of $h_0$  in \eqref {5116+},  \eqref{5117} and \eqref{5118}, respectively,  we see that $h_0$ is   uniformly bounded if and only if the following relation holds,
\begin{align}\label{uniform-relation}
\alpha_0(u_2)\,=\,1.
\end{align}
In the other words, in {Cases-3.2.1-3.2.3},
 $$\lim_{y\to y_0}h_0\,=\,-\infty,$$
 if (\ref{uniform-relation}) does not hold.
\end{rem}

\section{Reduced Ding functional and existence criterion}

By Lemma \ref{energy-preserve} and Theorem \ref{E1-legendre}, we see that for any $u\in{\mathcal  E}^1_{K\times K}(2P)$, its Legendre function
  $$\psi_u(x)\,=\,\sup_{y\in 2P}\{\langle x, y\rangle-u(y)\}\,\le\, v_{2P}(x)$$
  corresponds to  a $K\times K$-invariant weak K\"ahler potential $\phi_u=\psi_u-\psi_0$ which belongs to $\mathcal E^1_{K\times K}(M,-K_M)$. Here we can choose $\psi_0$ to be the Legendre function $\psi_{2P}$ of Guillemin function $u_{2P}$ as
  in (\ref{ Legendre-function}).
  As we know,  $e^{-\phi_u}\in L^p(\omega_0)$ for any $p\ge 0$. Thus $\int_{\mathfrak a_+}e^{-\psi_u}\mathbf J(x)dx$ is well-defined.

We introduce the following  functional   on $\mathcal E^1_{K\times K}(2P)$ by
  $$\mathcal D(u)\,=\,\mathcal L(u)+\mathcal F(u),$$
where
$$\mathcal L(u)\,=\,\frac1V\int_{2P_+}u\pi\,dy-u(4\rho)$$
and
$$\mathcal F(u)\,=\,-\log\left(\int_{\mathfrak a_+}e^{-\psi_u}\mathbf J(x)dx\right)+u(4\rho).$$
It is easy to see that on  a smooth Fano  compactification  of $G$,
$$\mathcal L(u_\phi)+u_\phi(4\rho)\,=\,-\frac1{(n+1)V}\sum_{k=0}^n\int_M\phi\omega_\phi^k\wedge\omega_0^{n-k}$$
and $\mathcal D(u_\phi)$ is just the Ding functional $F(\phi)$. We note that
a similar functional on  such  Fano  manifolds has been studied for Mabuchi solitons  in  \cite[Section 4]{LZ}).
Hence, for convenience, we call $\mathcal D(\cdot)$ the reduced Ding functional  on a $\mathbb Q$-Fano compactifications of $G$.

In this section, we will use the variation method to prove Theorem \ref{LTZ2}  by verifying the properness of $\mathcal D(\cdot)$. We assume that the moment  polytope $P$  is fine
so that the Ricci potential $h_0$ is uniformly bounded  above by Proposition \ref{ricci-potential}.

\subsection{A criterion for the properness of  $\mathcal D(\cdot)$}
 In this subsection, we establish a  properness criterion for $\mathcal D(u_\phi)$, namely,

\begin{prop}\label{ding-proper}   Let  $M$ be a $\mathbb Q$-Fano compactification of $G$. Suppose that the  moment  polytope $P$ is fine and it  satisfies \eqref{bary}. Then there are constants $\delta$ and $C_\delta$ such that
\begin{align}\label{proper-d}
\mathcal D(u)\,\geq\,  \delta \int_{2P_+}u\pi(y)\,dy+ C_\delta,~u\in\mathcal E^1_{K\times K}(2P).
\end{align}

\end{prop}

The proof goes almost the same as in \cite{LZ}. We sketch the arguments here for completeness.  First we note that $u_\phi$ satisfies the normalized condition
$u\geq u(O)=0$. Then we have the following  estimate for the linear term  $\mathcal L(\cdot)$ as in \cite[Proposition 4.5]{LZ}.

\begin{lemm}\label{linear-proper-thm}
Under the assumption (\ref{bary}), there exists a constant $\lambda>0$ such that
\begin{eqnarray}
\mathcal L(u)\,\geq\, \lambda\int_{2P_+}u\pi(y)\,dy,~\forall u\in {\mathcal  E}^1_{K\times K}(2P).
\end{eqnarray}
\end{lemm}

For the non-linear term $\mathcal F(\cdot)$, we can also get an analogy of \cite[Lemma 4.8]{LZ} as follows.

\begin{lemm}\label{non-linear-lem}
For any $\phi\in\mathcal E^1_{K\times K}(M,-K_M)$, let $$\tilde\psi_\phi:=\psi_\phi-4\rho_ix^i,~x\in\mathfrak a_+.$$
Then
\begin{eqnarray}\label{0406}
\mathcal F(u_\phi)\,=\,-\log\left(\int_{\mathfrak a_+}e^{-(\tilde\psi_\phi-\inf_{\mathfrak a_+}\tilde\psi_\phi)}\prod_{\alpha\in\Phi_+}\left(\frac{1-e^{-2\alpha_ix^i}}{2}\right)^2dx\right).
\end{eqnarray}
Consequently, for any $c>0$,
\begin{eqnarray}\label{0407}
\mathcal F(u_\phi)\,\geq\, \mathcal F\left(\frac{u_\phi}{1+c}\right)-n\cdot\log(1+c).
\end{eqnarray}
\end{lemm}

Let $\phi_0,\phi_1\in\mathcal E^1_{K\times K}(-K_M)$ and $u_0,u_1$ be two Legendre functions of $\psi_0+\phi_0$ and $\psi_0+\phi_1$, respectively.
Let $u_t$ $(t\in [0,1])$ be a linear path connecting $u_0$ to $u_1$ as in (\ref{linear-path}). Then by Theorem \ref{E1-legendre},
the corresponding Legendre functions  $\psi_t$ of $u_t$  give a path in $\mathcal E^1_{K\times K}(-K_M ,-K_M)$.  The following
lemma shows that $\mathcal F(\psi_t)$ is convex in $t$.

\begin{lemm}\label{F-convex} Let
$$\hat{\mathcal F}(t)\,=\,-\log\int_{\mathfrak a_+}e^{-\psi_t}\mathbf J(x)dx,~t\in [0,1].$$
Then $\hat{\mathcal F}(t)$ is convex in $t$ and so is $\mathcal F(\psi_t)$.
\end{lemm}

\begin{proof} By definition,
we have
\begin{align}\label{ex-cov-01}
\psi_t(tx_1+(1-t)x_0)=&\sup_{y}\{\langle y,tx_1+(1-t)x_0\rangle-(tu_1(y)+(1-t)u_0(y))\}\notag\\
\leq&t\sup_{y}\{\langle y,x_1\rangle-u_1(y)\}\notag\\&+(1-t)\sup_{y}\{\langle y,x_0\rangle-u_0(y))\}\notag\\
\leq& t\psi_1(x_1)+(1-t)\psi_0(x_0),~\forall  ~x_0,x_1\in\mathfrak a.
\end{align}
On the other hand,
$$\log\mathbf J(tx_1+(1-t)x_0)\geq t\log\mathbf J(x_1)+(1-t)\log\mathbf J(x_0),~\forall x_0,x_1\in\mathfrak a_+.$$
Combining these two inequalities, we get
$$(e^{-\psi_t}\mathbf J)(tx_1+(1-t)x_0)\geq(e^{-\psi_1}\mathbf J)^t(x_1)(e^{-\psi_0}\mathbf J)^{1-t}(x_0),~\forall x_0,x_1\in\mathfrak a_+.$$
Hence, by applying the  Prekopa-Leindler inequality to three functions $e^{-\psi_t}\mathbf J, e^{-\psi_1}\mathbf J$
and $e^{-\psi_0}\mathbf J$ (cf. \cite{Timergaliev}), we prove
$$-\log\int_{\mathfrak a_+}e^{-\psi_t}\mathbf J(x)dx\,\le\, -t\log\int_{\mathfrak a_+}e^{-\psi_1}\mathbf J(x)dx-(1-t)\log\int_{\mathfrak a_+}e^{-\psi_0}\mathbf J(x)dx.$$
This means that $\hat{\mathcal F}(t)$ is convex.
\end{proof}

\begin{proof}[Proof of Proposition  \ref{ding-proper}]
By Proposition  \ref{ricci-potential},
\begin{eqnarray*}
A(y)\,=\,{\frac{V}{\int_{\mathfrak a_+}e^{-\psi_0}\mathbf J(x)dx}}e^{h_0(\nabla u_0(y))}
\end{eqnarray*}
is bounded, where $y(x)=\nabla\psi_0(x)$. Then the functional
\begin{eqnarray*}
\mathcal D_A(u)\,=\,\mathcal L^0_A(u)+\mathcal F(u),
\end{eqnarray*}
is well-defined  on   $\mathcal E^1_{K\times K}(2P)$, where
$$\mathcal L^0_A(u)\,=\,{\frac1V}\int_{2P_+}u A(y)\pi(y)\,dy-u(4\rho).$$
It is easy to see that  that $u_0$ is a critical point of $\mathcal D_A(\cdot)$.
On the other hand, by Lemma \ref{F-convex}, $\mathcal F(\cdot)$ is convex along any path in $\mathcal E^1_{K\times K}(M,-K_M)$ determined by their Legendre functions  as in (\ref{linear-path}).
Note that  $\mathcal L_A^0(\cdot)$ is convex in $\mathcal E^1_{K\times K}(2P)$.  Hence
\begin{eqnarray*}
\mathcal D_A(u)\,\geq\,\mathcal D_A(u_0),~\forall u\in\mathcal E^1_{K\times K}(2P).
\end{eqnarray*}
Now together with Lemma \ref{linear-proper-thm} and Lemma \ref{non-linear-lem},  we can apply arguments in the proof of \cite[Proposition 4.9]{LZ} to proving that  there is a constant $C>0$ such that for any $u\in\mathcal E^1_{K\times K}(2P)$,
\begin{align}
\mathcal D(u)\,\geq\,{\frac{ C\lambda}{1+C}}\int_{2P_+}u\pi(y)\,dy+\mathcal D_A(u_0)-n\log(1+C).\notag
\end{align}
Therefore, we get (\ref{proper-d}).

\end{proof}

\subsection{Semi-continuity}

Write ${\mathcal E}^1_{K\times K}(2P)$ as
$${\mathcal E}^1_{K\times K}(2P)\,=\,\bigcup_{\kappa\geq0}{\mathcal E}^1_{K\times K}(2P;\kappa),$$
where
$${\mathcal E}^1_{K\times K}(2P;\kappa)=\{u\in{\mathcal E}^1_{K\times K}(2P)|\int_{2P_+}u\pi\,dy\leq\kappa\}.$$
By \cite[Lemma 6.1]{LZZ} and Fatou's lemma, it is easy to see that any sequence $\{u_n\}\subseteq {\mathcal E}^1_{K\times K}(2P;\kappa)$ has a subsequence which converges locally uniformly to some $u_\infty$ in it. Thus each ${\mathcal E}^1_{K\times K}(2P;\kappa)$, and so ${\mathcal E}^1_{K\times K}(2P)$ is complete.  Moreover, we have

\begin{prop}\label{lsc-prop}
The reduced Ding functional $\mathcal D(\cdot)$ is lower semi-continuous on the space ${\mathcal E}^1_{K\times K}(2P)$. Namely, for any sequence $\{u_n\}\subseteq\hat{\mathcal E}^1_{K\times K}(2P)$, which converges locally uniformly to some $u_\infty$, we have $u_\infty\in {\mathcal E}^1_{K\times K}(2P)$ and it holds
\begin{align}\label{lsc}
\mathcal D(u_\infty)\,\leq\,\liminf_{n\to\infty}\mathcal D(u_n).
\end{align}
\end{prop}

\begin{proof}
By Fatou's lemma, we have
\begin{align}\label{13203}
\int_{2P_+}u_\infty\pi\,dy\,\leq\,\liminf_{n\to+\infty}\int_{2P_+}u_n\pi\,dy<+\infty.
\end{align}
Then $u_\infty\in {\mathcal E}^1_{K\times K}(2P)$ and
$$\mathcal L(u_\infty)\,\leq\,\liminf_{n\to+\infty}\mathcal L(u_n).$$
It remains to estimate $\mathcal F(u_\infty)$.  Note that $u_\infty$ is finite everywhere in Int$(2P)$  by the locally uniformly convergence and its Legendre function $\psi_\infty\le v_{2P}$.   Thus, for any $\epsilon_0\in(0,1)$ there is a constant $M_{\epsilon_0}>0$ such that
(cf. \cite[Lemma 2.3]{Coman-Guedj-Sahin-Zeriahi}),
\begin{align}\label{13201}
\psi_\infty(x)\,\geq\,(1-\epsilon_0)v_{2P}(x)-M_{\epsilon_0},\forall x\in\mathfrak a.
\end{align}

On the other hand, the Legendre function $\psi_n$ of $u_n$ also converges locally uniformly to $\psi_\infty$. Then
$$\partial\psi_n\to\partial\psi_\infty$$
almost everywhere. Since
$$\psi_n(O)\,=\,\psi_\infty(O)\,=\,0,\forall n\in\mathbb N_+,$$
we have
\begin{align}\label{13202}
\psi_n(x)\,\geq\,(1-\epsilon_0)v_{2P}(x)-M_{\epsilon_0},\forall x\in\mathfrak a
\end{align}
as long as $n\gg1$. Note that
$$0\,\leq\,\mathbf J(x)\,\leq\, e^{4\rho(x)},\forall x\in\mathfrak a_+.$$
By choosing an $\epsilon_0$ such that $4\rho\in(1-\epsilon_0)\text{Int}(2P)$, we get
$$\int_{\mathfrak a_+}e^{M_{\epsilon_0}-(1-\epsilon_0)v_{2P}(x)}\mathbf J(x)dx<+\infty.$$
Hence, combining this with \eqref{13201} and \eqref{13202} and using Fatou's lemma, we derive
\begin{align*}
-\log\left(\int_{\mathfrak a_+}e^{-\psi_\infty}\mathbf J(x)dx\right)\,\leq\,\liminf_{n\to+\infty}\left[-\log\left(\int_{\mathfrak a_+}e^{-\psi_n}\mathbf J(x)dx\right)\right].
\end{align*}
Therefore,  we have proved \eqref{lsc} by \eqref{13203}.
\end{proof}

\subsection{Proof of Theorem \ref{LTZ2}}
Now we prove the sufficient  part of Theorem \ref{LTZ2}.
Suppose that \eqref{bary} holds. Then by Theorem \ref{ding-proper} and Proposition \ref{lsc-prop}, there is a minimizing sequence $\{u_n\}$ of $\mathcal D(\cdot)$ on ${\mathcal E}^1_{K\times K}(2P)$, which converges locally uniformly to some $u_\star\in{\mathcal E}^1_{K\times K}(2P)$ such that
\begin{align}\label{minimizer-in-E}
\mathcal D(u_\star)\le \lim_{u\in{\mathcal E}^1_{K\times K}(2P)}\mathcal D(u).
\end{align}
 Let $\psi_\star$ be the  Legendre function of $u_\star$. Then  by Theorem \ref{E1-legendre},  we have
 $$\phi_\star=\psi_\star-\psi_0\in{\mathcal E}^1_{K\times K}(M,-K_M).$$
 We need to show that $\phi_\star$ satisfies the K\"ahler-Einstein equation (\ref{singular-ke-equation}).

 \begin{prop}\label{minimizer-in-E-prop}
$\phi_\star$  satisfies the K\"ahler-Einstein equation \eqref{singular-ke-equation}.
\end{prop}

\begin{proof}
Let  $\{u_t\}_{t\in[0,1]}\subseteq{\mathcal E}^1_{K\times K}(2P)$ be a family convex functions with $u_0=u_\star$ and $\psi_t$ the corresponding Legendre functions   of $u_t$.  Then
by  Part (2) in  Lemma \ref{derivative-dual},
 $$\dot\psi_0\,=\,-\dot u_0,~{\rm almost ~everywhere}.$$
 Note that
 $$\int_{\mathfrak a_+}e^{-\psi_\star}\mathbf J(x)dx=V, $$
 Thus by  \eqref{full-mass} in Lemma \ref{energy-preserve}, we get
\begin{align}\label{variation-1}
\left.\frac d{dt}\right|_{t=0}\mathcal D(u_t)&=\frac1V\int_{2P_+}\dot u_0\pi\,dy+\frac{\int_{\mathfrak a_+}\dot\psi_0e^{-\psi_\star}\mathbf J(x)dx}{V}\notag\\
&=\frac1V\int_{\mathfrak a_+}\dot \psi_0[ e^{-\psi_\star}\mathbf J(x) -\text{MA}_{\mathbb R;\pi}(\psi_\star)] dx.
\end{align}

For any continuous, compactly supported  $W$-invariant   function $\eta\in C_0({\mathfrak a})$,  we consider a  family of functions $u_\star+t\eta$.  In general,  it may not be convex for $t\not=0$ since $u_\star$ is  just  weakly convex.  In the following, we use a trick to modify the function  $\mathcal D(u_t)$  as in \cite[Section 2.6]{Berman-Berndtsson}.
Define a family of  $W$-invariant functions  by
$$\hat\psi_t\,=\,\sup_{\phi\in\mathcal E^1_{K\times K}(M,-K_M)}\{\psi_\phi|\psi_\phi\leq\psi_\star+t\eta\}.$$
Then it is easy to see that the Legendre function $\hat u_t$ of $\hat\psi_t$ satisfies
$$|\hat u_t- u_{0}|\,\le\, C, ~\forall |t|\ll1.$$
By Theorem \ref{E1-legendre},
 we see that $(\hat\psi_t-\psi_0)\in\mathcal E^1_{K\times K}(M,-K_M)$.
Without loss of generality, we may assume that $\hat\psi_t$ satisfies (\ref{normalization}).

Let
$$\tilde{\mathcal D}(t)=\mathcal L(\hat u_t)+\mathcal F(\hat u_t).$$
Then
 \begin{align}\label{5120}
\tilde {\mathcal D}(0)\,=\,\mathcal D(u_\star)
\end{align}
and
\begin{align}\label{5119}
\tilde{\mathcal D}(t)
\geq\mathcal D(u_\star).
\end{align}

\begin{claim}\label{derivative}${\mathcal L}(\hat u_t)+\hat u_t(4\rho)$ is differentiable for $t$. Moreover,
\begin{align}\label{variation-1}\left.\frac{d}{dt}\right|_{t=0}( {\mathcal L}(\hat u_t)+\hat u_t(4\rho))
=-\frac1V \int_{M}\eta \omega_{\phi_\star}^n.\end{align}
\end{claim}

To prove this claim, we let a convex function $g(t)=\hat u_t(p)$ for  each fixed $y\in2P$.  Then it  has left and right derivatives $g'_-(t;p),g'_+(t;p)$, respectively.  Moreover, they are monotone and $g'_-(t;p)\leq g'_+(t;p)$. Thus,  $g'_-,g'_+\in L^\infty_{\text{loc}}$.  It follows that
$$\left.\frac d{dt}\right|_{t=\tau^{\pm}}\int_{2P_+}\hat u_{\tau}\pi dy=\lim_{\tau'\to0^\pm}\frac1{\tau'}\int_{2P_+}(\hat u_{\tau+\tau'}-\hat u_\tau)\pi dy$$
  and by the Lebesgue monotone convergence theorem,
$$\left.\frac d{dt}\right|_{t=\tau^{\pm}}\int_{2P_+}\hat u_{\tau}\pi dy\,=\,\int_{2P_+}g'_\pm(\tau;p)\pi dy.$$
Recall that $g'_-(t;p)=g'_+(t;p)$ holds almost everywhere. Thus we see that
$$\mathcal L(\hat u_t)+u(4\rho)=\frac1V\int_{2P_+}\hat u_t\pi dy$$
is differentiable.

Note that
$$u_{\hat\psi_t}\,=\,u_{\psi_\star+t\eta}, $$
where  $u_{\psi_\star+t\eta}$ is the  Legendre function of   ${\psi_\star+t\eta}$.  It follows from Part (2) in  Lemma \ref{derivative-dual} that
$$\dot{\hat\psi}_0\,=\,-\dot u_0=\eta,~{\rm almost ~everywhere}. $$
Hence by   Lemma \ref{approximate-lemma} (or Remark \ref{integral-formula}),   we get
\begin{align}
  \left.\frac{d}{dt}\right|_{t=0}( {\mathcal L}(\hat u_t)+\hat u_t(4\rho))
&=\frac1V\int_{2P_+}\dot{\hat u}_0\pi dy\notag\\
&=-\frac1V\int_{2P}\eta \pi dy=-\frac1V\int_{{\mathfrak a_+}} \eta \text{MA}_{\mathbb R;\pi}(\psi_0)dx  \notag\\
&=-\frac1V\int_{M}\eta \omega_{\phi_\star}^n,\notag
\end{align}
where $\phi_\star=\psi_*-\psi_0$.
The claim is proved.

Similar  to Claim \ref{derivative}, we have
\begin{align}\label{variation-2}\left.\frac{d}{dt}\right|_{t=0}(\mathcal F(\hat u_t)-\hat u_t(4\rho))&=\frac1V\int_{\mathfrak a_+}\eta e^{-\psi_\star}\mathbf J(x)dx\notag\\
&=\int_G \eta e^{-\phi_\star+h_0}\omega_0^n.
\end{align}
Thus, by \eqref{5120}-\eqref{variation-2}, we derive
\begin{align}
0\,=\,\left.\frac d{dt}\right|_{t=0}\tilde{\mathcal D}(t)&=\frac1V\int_{G}\eta[ e^{-\phi_\star+h_0}\omega_0^n - \omega_{\phi_\star}^n ] dx.\label{vari-psi-star}
\end{align}
As a consequence,
 $$ \omega_{\phi_*}^n \,=\,e^{-\psi_\star+h}\omega_0^n,~{\rm in}~G.$$
Therefore, by Lemma \ref{approximate-lemma} and $KAK$-integration formula,
we prove  that $\phi_\star$  satisfies (\ref{singular-ke-equation}) on $G$.

Next we show  that $\omega_{\phi_\star}$ can be extended as a singular K\"ahler-Einstein metric on $M$.
Choose an $\epsilon_0$ such that $4\rho\in\text{Int}(2(1-\epsilon)P)$. Since $u_\star\in{\mathcal E}^1_{K\times K}(2P)$,
by Lemma \ref{energy-preserve}, there is a constant $C_\star>0$ such that
$$\psi_\star\,\geq\,(1-\epsilon_0)v_{2P}-C_\star.$$
Thus
$$e^{-\psi_\star(x)}\mathbf J(x)$$
is bounded on $\mathfrak a_+$.
Also $\pi(\partial\psi_\star)$ is bounded. Therefore, by  \eqref{full-mass}, for any $\epsilon>0$, we can find a neighborhood $U_\epsilon$ of $M\setminus G$ such that
$$\left|\int_{U_\epsilon}(\omega_{\phi_\star}^n-e^{h_0-\psi_\star}\omega_0^n)\right|\,<\,\epsilon.$$
This implies that  $\phi_\star$ can be extended to be  a global solution of  (\ref{singular-ke-equation}) on $M$. The proposition is proved.
\end{proof}

\section{$\mathbb Q$-Fano compactification of ${SO}_4(\mathbb C)$}

In this section, we will construct $\mathbb Q$-Fano compactifications of ${SO}_4(\mathbb C)$ as examples and in particular, we will prove  Theorem \ref{LTZ3}. Note that in this case rank$(G)=2$.  Thus we can use Theorem \ref{LTZ2} to verify  whether there exists   a  K\"ahler-Einstein metric on a $\mathbb Q$-Fano ${SO}_4(\mathbb C)$-compactification by computing the barycenter of their moment polytopes $P_+$. For convenience, we will work with $P_+$ instead of $2P_+$ throughout this section.   Then it is easy to see that the existence criterion \eqref{bary} is equivalent to

\begin{align}\label{bary-p}
bar(P_+)\in 2\rho+\Xi.
\end{align}

Denote
$$R(t)\,=\,\left(\begin{aligned}&\cos t&-\sin t\\&\sin t&\cos t\end{aligned}\right).$$
Consider the canonical embedding of $SO_4(\mathbb C)$ into $GL_4(\mathbb C)$ and choose the maximal torus
$$T^\mathbb C\,=\,\left\{\left(\begin{aligned}&R(z^1)&O\\&O&R(z^2)\end{aligned}\right)|z^1,z^2\in\mathbb C\right\}.$$
Choose the basis of $\mathfrak N$ as $E_1,E_2$, which generates  the $R(z^1)$ and $R(z^2)$-action. Then we have two  positive roots in $\mathfrak M$,
$$\alpha_1=(1,-1),\alpha_2=(1,1).$$
Also  we have
$$\mathfrak a^*_+=\{(x,y)|-x\leq y\leq x\}, ,~2\rho=(2,0)$$
and
\begin{align}\label{xi}
2\rho+\Xi=\{(x,y)|-x+2\leq y\leq x-2\}.
\end{align}

\subsection{Gorenstein Fano $SO_4(\mathbb C)$-compactifications}\label{sect-6-1}

In this subsection, we  use Lemma \ref{polytope-coefficient} to exhaust all polytopes associated to Gorenstein Fano compactifications. Here by Gorenstein, we mean
that $K_{M_{reg}}^{-1}$ can be extended as a holomorphic vector line bundle on $M$. In this case, the whole polytope $P$ is a lattice polytope. Also, since $2\rho=(2,0)$, each outer edge \footnote {An edge of $P_+$ is called an outer one if it does not lie in any Weyl wall, cf. \cite{LZZ}.} of $P_+$ must lies on some line
\begin{align}\label{line-pq}l_{p,q}(x,y)\,=\,(1+2p)-(px+qy)\,=\,0
\end{align}
for some coprime pair $(p,q)$. Assume that $l_{p,q}\geq 0$ on $P$. By convexity and $W$-invariance of $P$, $(p,q)$ must satisfy $$p\geq |q|\geq0.$$

Let us start at the outer edge $F_1$ of $P_+$ which intersects the Weyl wall $$W_1=\{x-y=0\}.$$
There are two cases: \emph{Case-1.} $F_1$ is orthogonal to $W_1$; \emph{Case-2.} $F_1$ is not orthogonal to $W_1$.

\emph{Case-1.}   $F_1$ is orthogonal to $W_1$. Then $F_1$  lies on
$$\{(x,y)|~l_{1,1}(x,y)=3-x-y=0\}.$$
Consider the vertex $A_1=(x_1,3-x_1)$ of $P_+$ on this edge and suppose that the other edge $F_2$ at this point lies on
$$\{(x,y)|~l_{p_2,q_2}(x,y)=0\}.$$
Thus
\begin{align}\label{F-2-1}2p_2+1=x_1p_2+(3-x_1)q_2,\end{align}
and by convexity of $P$,
  $$p_2>q_2\geq0.$$

We will have two subcases according to the possible choices $A_1=(2,1)$ or $(3,0)$.

\emph{Case-1.1.}  $A_1=(2,1)$.  Then by (\ref{F-2-1}),
   $$2p_2+1=2p_2+q_2.$$
Thus $q_2=1$ and $p_2\geq2$.

On the other hand, $l_{p_2,q_2}$ must pass another lattice point $A_2=(x_2,y_2)$ as the other endpoint of $F_2$. It is direct to see that there are only two possible choices $p_2=2,4$ and three choices of $A_2=(5,-5)$, $(3,-1)$ and $(3,-3)$.

\emph{Case-1.1.1.} If $A_2=(5,-5)$ which lies on the other Weyl wall $W_2=\{x+y=0\}$. There can not be any other outer edges of $P_+$, and $P_+$ is given by Figure (7-1-1).
\begin{figure}[h]
\begin{center}
\begin{tikzpicture}
\draw [dotted] (0,-5) grid[xstep=1,ystep=1] (5,2);
\draw (0,0) node{$\bullet$};
\draw (2,0) node{$\bullet$};
\draw (1.6,0.3) node{$2\rho$};
\draw [semithick] (0,0) --(1.5,1.5) -- (2,1) -- (5,-5) -- (0,0);
\draw (2.2,1.2) node{$A_1$};
\draw (5.2,-4.8) node{$A_2$};
\draw [very thick, -latex] (0,0) -- (1,-1);
\draw [very thick, -latex] (0,0) -- (1,1);
\draw (0.2,2.2) node{(7-1-1)};
\end{tikzpicture}
\begin{tikzpicture}
\draw [dotted] (0,-3) grid[xstep=1,ystep=1] (3,2);
\draw (0,0) node{$\bullet$};
\draw (2,0) node{$\bullet$};
\draw (1.6,0.3) node{$2\rho$};
\draw [semithick] (0,0) -- (1.5,1.5) -- (2,1) -- (3,-1) -- (3,-3) -- (0,0);
\draw (2.2,1.2) node{$A_1$};
\draw (3.2,-0.8) node{$A_2$};
\draw (3.2,-2.8) node{$A_3$};
\draw [very thick, -latex] (0,0) -- (1,-1);
\draw [very thick, -latex] (0,0) -- (1,1);
\draw (0.2,2.2) node{(7-1-2)};
\end{tikzpicture}
\end{center}
\end{figure}
By Theorem \ref{LTZ2} (or equivalently \eqref{bary-p}), this compactification admits no K\"ahler-Einstein metric.

\emph{Case-1.1.2.} $A_2=(3,-1)$. Then we exhaust the third edge $F_3$ which lies on $$l_{p_3,q_3}=2p_3+1-p_3x-q_3y,$$
so that
\begin{align*}
2p_3+1&=3p_3-q_3,\\
p_3&>2q_3\geq0.
\end{align*}
Hence the only possible choice is $p_3=1, q_3=0$ and the other endpoint of $F_3$ is $A_3=(0,-3)$.
Then $P_+$ is given by Figure (7-1-2).
Again, this compactification admits no K\"ahler-Einstein metric.

\emph{Case-1.1.3.} If $A_2=(3,-3)$ which lies on the other Weyl wall $W_2=\{x+y=0\}$. There can not be any other outer edges of $P_+$, and $P_+$ is given by Figure (7-1-3).
By Theorem \ref{LTZ2}, this compactification admits no K\"ahler-Einstein metric.

\emph{Case-1.2.}  $A_1=(3,0)$.  By the same exhausting progress as in \emph{Case-1.1.} There  are two possible polytopes $P_+$,  \emph{Case-1.2.1} and \emph{Case-1.2.2} (see Figure (7-1-4) and Figure (7-1-5)).

\begin{figure}[h]
\begin{center}
\begin{tikzpicture}
\draw [dotted] (0,-3) grid[xstep=1,ystep=1] (3,2);
\draw (0,0) node{$\bullet$};
\draw (2,0) node{$\bullet$};
\draw (1.6,0.3) node{$2\rho$};
\draw [semithick] (0,0) --(1.5,1.5) -- (2,1) -- (3,-3) -- (0,0);
\draw (2.2,1.2) node{$A_1$};
\draw (3.2,-2.8) node{$A_2$};
\draw [very thick, -latex] (0,0) -- (1,-1);
\draw [very thick, -latex] (0,0) -- (1,1);
\draw (0.2,2.2) node{(7-1-3)};
\end{tikzpicture}
\begin{tikzpicture}
\draw [dotted] (0,-3) grid[xstep=1,ystep=1] (3,2);
\draw (0,0) node{$\bullet$};
\draw (2,0) node{$\bullet$};
\draw (1.6,0.3) node{$2\rho$};
\draw [semithick] (0,0) -- (1.5,1.5) -- (3,0) -- (3,-3) -- (0,0);
\draw (1.7,1.8) node{$A_1$};
\draw (3.2,0.2) node{$A_2$};
\draw (3.2,-2.8) node{$A_3$};
\draw [very thick, -latex] (0,0) -- (1,-1);
\draw [very thick, -latex] (0,0) -- (1,1);
\draw (0.2,2.2) node{(7-1-4)};
\end{tikzpicture}
\end{center}
\end{figure}

 \emph{Case-1.2.1.}  This compactification admits no K\"ahler-Einstein metric.

\emph{Case-1.2.2.}
This compactification admits a K\"ahler-Einstein metric.

\emph{Case-2.}   $F_1$ is not orthogonal to $W_1$.  Then its intersection $A_1=(x_1,x_1)$ with $W_1$ is a vertex of $P$.
We see that $F_1$ lies on $l_{p_1,q_1}$ and
\begin{align*}
2p_1&=(p_1+q_1)x_1,\\
p_1&>q_1\geq0,\\
x_1&=2+\frac{1-2q_1}{p_1+q_1}\in\mathbb N_+.
\end{align*}
So the only choice is
$$p_1=1,q_1=0$$
and $A_1=(3,1)$. The only new polytope $P_+$ is given by Figure (7-1-6),
\begin{figure}[h]
\begin{center}
\begin{tikzpicture}
\draw [dotted] (0,-2) grid[xstep=1,ystep=1] (3,2);
\draw (0,0) node{$\bullet$};
\draw (2,0) node{$\bullet$};
\draw (1.6,0.3) node{$2\rho$};
\draw [semithick] (0,0) -- (1.5,1.5) -- (3,0) -- (1.5,-1.5) -- (0,0);
\draw (1.7,1.8) node{$A_1$};
\draw (3.2,0.2) node{$A_2$};
\draw (1.7,-1.8) node{$A_3$};
\draw [very thick, -latex] (0,0) -- (1,-1);
\draw [very thick, -latex] (0,0) -- (1,1);
\draw (0.2,2.2) node{(7-1-5)};
\end{tikzpicture}
\begin{tikzpicture}
\draw [dotted] (0,-3) grid[xstep=1,ystep=1] (3,3);
\draw (0,0) node{$\bullet$};
\draw (2,0) node{$\bullet$};
\draw (1.6,0.3) node{$2\rho$};
\draw [semithick] (3,3) -- (0,0) -- (3,-3) -- (3,3);
\draw (3.2,2.8) node{$A_1$};
\draw (3.2,-2.8) node{$A_2$};
\draw [very thick, -latex] (0,0) -- (1,-1);
\draw [very thick, -latex] (0,0) -- (1,1);
\draw (0.2,2.2) node{(7-1-6)};
\end{tikzpicture}
\end{center}
\end{figure}
which admits K\"ahler-Einstein metric.

It is known that \emph{Case-1.1.2}, \emph{Case-1.2.1} and \emph{Case-2} are the only smooth $SO_4(\mathbb C)$-compactifications as shown in \cite{LTZ}. We summarize results of this subsection in Table-1.

\begin{table}[t]
\begin{tabular}{|l|l|l|l|l|}
\hline
{\rm No.} & Edges, except Weyl walls & Volume                               & KE? & Smoothness \\ \hline
(7-1-1)   & 3-x-y=0; 5-2x-y=0        & $\frac{411}4$         & No  & Singular   \\ \hline
(7-1-2)   & 3-x-y=0; 5-2x-y=0; 3-x=0 & $\frac{10751}{180}$ & No  & Smooth     \\ \hline
(7-1-3)   & 3-x-y=0; 9-4x-y=0        & $\frac{16349}{972}$ & No  & Singular   \\ \hline
(7-1-4)   & 3-x-y=0; 3-x=0           & $\frac{1701}{20}$   & No  & Smooth     \\ \hline
(7-1-5)   & 3-x-y=0; 3-x+y=0         & $\frac{81}2$          & Yes & Singular   \\ \hline
(7-1-6)   & 3-x=0                    & $\frac{648}5$         & Yes & Smooth     \\ \hline
\end{tabular}
\caption{Gorenstein Fano $SO_4(\mathbb C)$-compactifications.}
\end{table}

\subsection{$\mathbb Q$-Fano $SO_4(\mathbb C)$-compactifications}\label{sect-6-2}

In general, for a fixed  integer $m>0$, it may be hard to give a classification  of  all  $\mathbb Q$-Fano compactifications such that $-mK_X$ is Cartier. This is because when $m$ is sufficiently divisible, there will be too many repeated polytopes directly using Lemma \ref{polytope-coefficient}.
To avoid this problem, we give a way to exhaust all $\mathbb Q$-Fano polytopes according to the intersection point of $\partial P_+$ with $x$-axis.

We will adopt the notations from the previous subsection.  We consider the intersection of $P_+$ with the positive part of the $x$-axis, namely $(x_0,0)$. Then
$$x_0\,=\,2+\frac1{p_0}$$ for some $p_0\in\mathbb N_+$, and there is an edge which lies on some $\{l_{p_0,q_0}=0\}$. Without loss of generality, we may   also assume that $\{l_{p_0,q_0}=0\}\cap\{y>0\}\not=\emptyset$.  Thus by symmetry, it suffices to consider the case
\begin{align*}p_0\,\geq\, q_0\,\geq\,0.\end{align*}
Indeed, by the prime condition, $q_0\not=0,\pm p_0$ if $p_0\not=1$.  Hence,  we may assume
\begin{align}\label{p0q0}
p_0\,>\, q_0\,>\,0,p_0\,\geq\,2.
\end{align}

We associate this number $p_0$ to each $\mathbb Q$-Fano polytope $P$ (and hence $\mathbb Q$-Fano compactifications of $SO_4(\mathbb C)$). By the convexity, other edges determined by  $l_{p,q}$ must satisfy (see the figure below)
$$p\,\leq\, p_0,$$
since we assume that \begin{align}\label{triangle-bound}P_+\subseteq(\{l_{p_0,q_0}\geq0\}\cap\mathfrak a_+).\end{align}
\begin{figure}[h]
\begin{center}
  \includegraphics[width=1.5in]{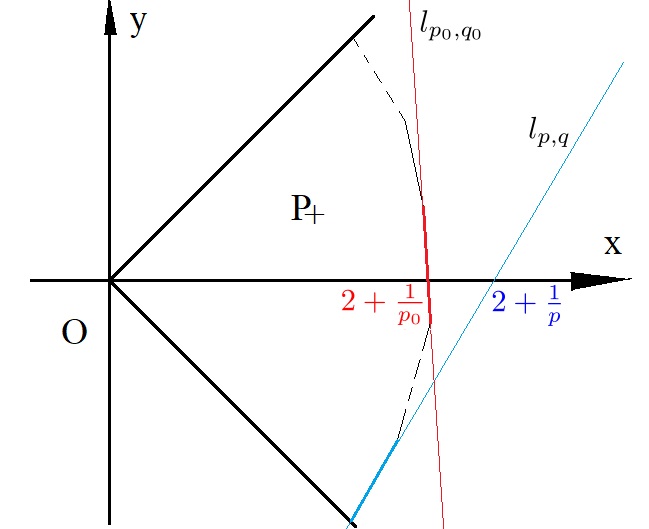}\\
\end{center}
\end{figure}
Thus, once $p_0$ is fixed, there are only finitely possible $\mathbb Q$-Fano compactifications of $SO_4(\mathbb C)$ associated to it. In the following table, we list all possible $\mathbb Q$-Fano compactifications with $p_0\leq2$. We also test the existence of K\"ahler-Einstein metrics on these compactifications.
In the appendix we list the nine non-smooth examples above labeled as in Table-2.

\begin{table}[t]
\begin{tabular}{|c|c|c|c|c|c|}
\hline
{\rm No.} & $p_0$      & $(p,q)$ of edges, except Weyl walls & Volume                 & KE? & Smoothness/Multiple \\ \hline
(1)       & \textbf{1} & $(1,0)$                             & $\frac{648}{5}$       & Yes & Smooth              \\ \hline
(2)       &            & $(1,0),(1,1)$                       & $\frac{1701}{20} $    & No  & Smooth              \\ \hline
(3)       &            & $(1,-1),(1,1)$                      & $\frac{81}{2}   $     & Yes & Multiple=1          \\ \hline
(4)       & \textbf{2} & $(2,1)$                             & $\frac{25000}{243}$   & No  & Multiple=3          \\ \hline
(5)       &            & $(2,1),(1,1)$                       & $\frac{411}{4}   $    & No  & Multiple=1          \\ \hline
(6)       &            & $(1,0),(2,1)$                       & $\frac{72728}{1215}$  & No  & Multiple=3          \\ \hline
(7)       &            & $(2,1),(1,-1)$                      & $\frac{947}{36} $     & No  & Multiple=3          \\ \hline
(8)       &            & $(2,-1),(2,1)$                      & $\frac{165625}{7776}$ & No  & Multiple=6          \\ \hline
(9)       &            & $(2,1),(1,0),(1,1)$                 & $\frac{10751}{180}$   & No  & Smooth              \\ \hline
(10)      &            & $(2,1),(1,-1),(1,1)$                & $\frac{12721}{486}$   & No  & Multiple=1          \\ \hline
(11)      &            & $(2,1),(2,-1),(1,1)$                & $\frac{164609}{7776}$ & No  & Multiple=6          \\ \hline
(12)      &            & $(2,1),(2,-1),(1,1),(1,-1)$         & $\frac{6059}{288}$    & No  & Multiple=6          \\ \hline
\end{tabular}
\caption{$\mathbb Q$-Fano $SO_4(\mathbb C)$-compactifications of cases $p_0\leq2$.}
\end{table}

\subsection{Proof of Theorem  \ref{LTZ3}}

\begin{proof}
We introduce some notations for convenience:  For any domain $\Omega\subset\overline{\mathfrak a_+^*}$,   define
$$\begin{aligned}
\text{Vol}(\Omega)&:=\int_\Omega\pi dx\wedge dy,\\\bar x(\Omega)&:=\frac1{V(\Omega)}\int_\Omega x\pi dx\wedge dy,\\\bar y(\Omega)&:=\frac1{V(\Omega)}\int_\Omega y\pi dx\wedge dy,\end{aligned}$$
and
 $$\bar c(\Omega):=\bar x+\bar y.$$
 By Theorem  \ref{LTZ2}  and \eqref{xi},   we have $\bar c(P_+)\geq2$ whenever the $\mathbb Q$-Fano compactification  of  $SO_4(\mathbb C)$ admits a  K\"ahler-Einstein metric.

Recall the number $p_0, q_0$ introduced in Section \ref{sect-6-2}. By \eqref{p0q0}, it is direct to see that for any $t\geq0$ such that $P_+$ intersects with $\{y=x-2t\}$,
\begin{align}\label{bar-P}
\bar c(P_+)&\leq\bar c(P_+\cap\{y\geq x-2t\})\notag\\
&\leq\bar c(\{l_{p_0,q_0}\geq0,0\leq x-y\leq 2t,y\geq-x\}).
\end{align}

\begin{figure}[h]
\begin{center}
  \includegraphics[width=1.5in]{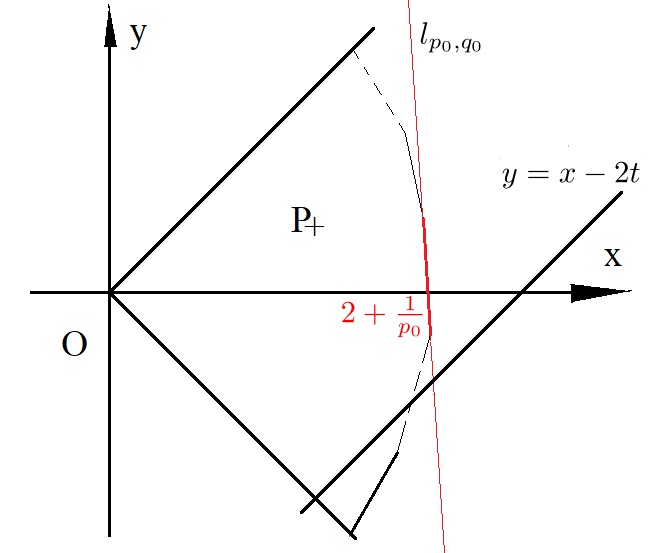}\\
\end{center}
\end{figure}

By a direct computation, we have
$$\begin{aligned}&\bar c(\{l_{p_0,q_0}\geq0,0\leq x-y\leq 2t,y\geq-x\})\\=&\frac3{35}\left(15kt+16b+\frac{3 b (10 b^2 + 10 b k t + 3 k^2 t^2)}{20 b^3 + 45 b^2 k t + 36 b k^2 t^2 + 10 k^3 t^3}\right),\end{aligned}$$
where $k=\frac{q_0-p_0}{p_0+q_0}$ and $b=\frac{2p_0+1}{p_0+q_0}$. Under the condition \eqref{p0q0},  by using software \texttt{Wolfram Mathematica 8}, we get
\begin{align}\label{bar-c}
\bar c(\{l_{p_0,q_0}\geq0,0\leq x-y\leq 2t,y\geq-x\})\leq\frac32b=\frac{6p_0+3}{2p_0+2q_0}.
\end{align}
On the other hand, a polytope with K\"ahler-Einstein metrics must satisfy
$$\bar c(P_+)>2.$$
Thus by \eqref{bar-P} and \eqref{bar-c}, we derive
\begin{align}\label{p-q-relation}q_0<\frac12p_0+\frac34.
\end{align}

By (\ref{p-q-relation}), we have
\begin{align}\label{p0-q0-relation}\text{Vol}(P_+)&\leq\text{Vol}(\{l_{p_0,q_0}\geq0,x\geq y\geq-x\})\notag\\
    &=\frac{8(1 + 2 p_0)^6}{45 (p_0^2 - q_0^2)^3}\notag\\
    &\leq\frac{8(1 + 2 p_0)^6}{45 (p_0^2 - ((1/2) p_0 + (3/4))^2)^3}.
    \end{align}
It turns that  for $p_0\geq9 $,
$$\text{Vol}(P_+)\leq\frac{224755712}{4100625}.$$
However,
 $$\text{Vol}(P^{(2)}_+)=\frac{1701}{20}>\text{Vol}(P^{(3)}_+)=\frac{10751}{180}>\frac{224755712}{4100625},$$
where $\text{Vol}(P^{(2)}_+)$ and $\text{Vol}(P^{(3)}_+)$ are volumes of polytopes in  \emph{Case-1.1.2} and \emph{Case-1.2.1}, respectively.
Thus there is no desired K\"ahler-Einstein polytope with its volume equal to ${\rm Vol}(P^{(2)}_+)$ or ${\rm Vol}(P^{(3)}_+)$ when $p_0\geq9$.

Since $q_0\in\mathbb N$, we can improve (\ref{p0-q0-relation}) to
$$\begin{aligned}\text{Vol}(P_+)
    &\leq\frac{8(1 + 2 p_0)^6}{45 (p_0^2 - [(1/2) p_0 + (3/4)]^2)^3}.\end{aligned}$$
Here  $[x]=\max_{n\in\mathbb Z}\{n\leq x\}$.  By the above estimation, when $p_0=4,6,7,8$, we have
\begin{align}\label{vol-p-compare}
\text{Vol}(P^{(2)}_+)\,>\,\text{Vol}(P^{(3)}_+)\,>\,\text{Vol}(P_+).
\end{align}
Hence, it remains to deal with the cases when $p_0=3,5$.  In these two cases, we shall rule out polytopes that may not satisfy \eqref{vol-p-compare}.

When $p_0=5$, there are three possible choices of $q_0$, i.e. $q_0=1,2,3$  by (\ref{p-q-relation}).  It is easy to see that \eqref{vol-p-compare} still holds for the first two cases by the second relation in (\ref{p0-q0-relation}).  Thus we only need to consider all possible polytopes when $q_0=3$.  In this case, $\{l_{5,3}=0\}$ is an edge of $P_+$.

\emph{Case-7.3.1.} $P_+$ has only one outer face which lies on $\{l_{5,3}=0\}$.  Then
$$\text{Vol}(P_+)\,=\,\frac{1771561}{23040}.$$

\emph{Case-7.3.2.} $P_+$ has two outer edges. Assume that the second one lies on $\{l_{p_1,q_1}=0\}$.  Then $$|q_1|\leq p_1\leq4\text{ or }p_1=5,q_1=-3.$$
By  a direct computation, we see that  \eqref{vol-p-compare} holds except the following  two subcases:

\emph{Case-7.3.2.1.} $p_1=4, q_1=3$,
$$\text{Vol}(P_+)\,=\,\frac{383478671}{5000940}.$$


\emph{Case-7.3.2.2.} $p_1=2, q_1=1$,
$$\text{Vol}(P_+)\,=\,\frac{567779}{7680}.$$


\emph{Case-7.3.3.}  $P_+$ has three outer edges.  Then  $P_+$ is obtained by  cutting one of  polytopes in \emph{Case-7.3.2} with adding  new edge $\{l_{p_2,q_2}=0\}$.  In fact we only need to consider $P_+$  obtained by cutting \emph{Case-7.3.2.1} and \emph{Case-7.3.2.2} above, since it  obviously satisfies  \eqref{vol-p-compare} in the other cases.  By our construction, we can assume that $|q_2|\leq p_2\leq p_1$. The only possible $P$ which does not satisfy \eqref{vol-p-compare} is the case that   $p_1=4, q_1=3$ and $p_2=2, q_2=1$.  However,
$$\text{Vol}(P_+)\,=\,\frac{92167583}{1250235}.$$

\emph{Case-7.3.4.} $P_+$  has four outer edges.  We only need to consider    $P_+$ which  is obtained  by  cutting \emph{Case-7.3.3} with adding  new edge  $\{l_{p_3,q_3}=0\}$ with $|q_3|\leq p_3\leq 2$.  One can show that all of  these  possible $P_+$ satisfy \eqref{vol-p-compare}.  Thus we do not need to consider more polytopes with more than four outer edges  in case  of $p_0=5$.   Hence we conclude that for all polytopes $P$ with $p_0=5$,
$$\text{Vol}(P_+)\neq {\rm Vol}(P^{(2)}_+)~{\rm or}~ {\rm Vol}(P^{(3)}_+).$$
  Theorem  \ref{LTZ3} is true when $p_0=5$.

The case $p_0=3$ can be ruled out in a same way. We only list the exceptional  polytopes  such that  the volumes  of $P_+$  do not satisfy \eqref{vol-p-compare}:

\emph{Case-7.3.1'.} $P_+$ has only one outer face $\{l_{3,2}=0\}$.  Then
$$\text{Vol}(P_+)\,=\,\frac{941192}{5625}.$$

\emph{Case-7.3.2'.} $P_+$ has two outer face $\{l_{3,2}=0\}$ and $\{l_{2,1}=0\}$.  Then
$$\text{Vol}(P_+)\,=\,\frac{177064}{1875}.$$


In summary,  when $p_0\ge 3$, the volume of $P_+$ is not equal to 
either ${\rm Vol}(P^{(2)}_+)$ or ${\rm Vol}(P^{(3)}_+)$.
Finally by exhausting all possible compactifications for $p_0=1,2$ (see Table-2), we finish the proof of  Theorem  \ref{LTZ3}.

\end{proof}

\begin{rem}
If $P_+$ is further symmetric under the reflection with respect to the $x$-axis, it is easy to see its barycenter is $(\bar x(P_+),0)$ and
$$\bar x(P_+)\leq\bar x(\{-x\leq y\leq x, 0\leq x\leq(2+\frac1{p_0})\})=\frac67(2+\frac1{p_0}).$$
Thus a K\"ahler-Einstein polytope of this type must satisfy $$p_0\leq3.$$
\end{rem}

\subsection{Appendix: Non-smooth $\mathbb Q$-Fano $SO_4(\mathbb C)$-compactifications with $p_0\leq2$}

In this appendix we list all polytopes $P_+$ of non-smooth $\mathbb Q$-Fano $SO_4(\mathbb C)$-compactifications with $p_0\leq2$, namely, (3)-(7) and (10)-(12) labeled as in Table-2.\\
\begin{figure}[h]
\begin{tikzpicture}
\draw [dotted] (0,-2) grid[xstep=1,ystep=1] (3,2);
\draw (0,0) node{$\bullet$};
\draw (2,0) node{$\bullet$};
\draw (1.6,0.3) node{$2\rho$};
\draw [semithick] (1.5,1.5) -- (0,0) -- (1.5,-1.5) -- (3,0) -- (1.5,1.5);
\draw (0.2,0.7) node{(3)};
\draw [very thick, -latex] (0,0) -- (1,-1);
\draw [very thick, -latex] (0,0) -- (1,1);
\end{tikzpicture}
\end{figure}
\begin{figure}[h]
\begin{tikzpicture}
\draw [dotted] (0,-5) grid[xstep=1,ystep=1] (5,2);
\draw (0,0) node{$\bullet$};
\draw (2,0) node{$\bullet$};
\draw (1.6,0.3) node{$2\rho$};
\draw [semithick] (0,0) -- (1.66,1.66) -- (5,-5) -- (0,0);
\draw (0.2,0.7) node{(4)};
\draw [very thick, -latex] (0,0) -- (1,-1);
\draw [very thick, -latex] (0,0) -- (1,1);
\end{tikzpicture}
\begin{tikzpicture}
\draw [dotted] (0,-5) grid[xstep=1,ystep=1] (5,2);
\draw (0,0) node{$\bullet$};
\draw (2,0) node{$\bullet$};
\draw (1.6,0.3) node{$2\rho$};
\draw [semithick] (0,0) -- (1.5,1.5) -- (2,1) -- (5,-5) -- (0,0);
\draw (0.2,0.7) node{(5)};
\draw [very thick, -latex] (0,0) -- (1,-1);
\draw [very thick, -latex] (0,0) -- (1,1);
\end{tikzpicture}
\end{figure}
\begin{figure}[h]
\begin{tikzpicture}
\draw [dotted] (0,-3) grid[xstep=1,ystep=1] (3,2);
\draw (0,0) node{$\bullet$};
\draw (2,0) node{$\bullet$};
\draw (1.6,0.3) node{$2\rho$};
\draw [semithick] (0,0) -- (1.66,1.66) -- (3,-1) -- (3,-3) -- (0,0);
\draw (0.2,0.7) node{(6)};
\draw [very thick, -latex] (0,0) -- (1,-1);
\draw [very thick, -latex] (0,0) -- (1,1);
\end{tikzpicture}
\begin{tikzpicture}
\draw [dotted] (0,-2) grid[xstep=1,ystep=1] (3,2);
\draw (0,0) node{$\bullet$};
\draw (2,0) node{$\bullet$};
\draw (1.6,0.3) node{$2\rho$};
\draw [semithick] (0,0) -- (1.66,1.66) -- (2.66,-0.34)  -- (1.5,-1.5) -- (0,0);
\draw (0.2,0.7) node{(7)};
\draw [very thick, -latex] (0,0) -- (1,-1);
\draw [very thick, -latex] (0,0) -- (1,1);
\end{tikzpicture}
\begin{tikzpicture}
\draw [dotted] (0,-2) grid[xstep=1,ystep=1] (3,2);
\draw (0,0) node{$\bullet$};
\draw (2,0) node{$\bullet$};
\draw (1.6,0.3) node{$2\rho$};
\draw [semithick] (0,0) -- (1.66,1.66) -- (2.5,0) -- (1.66,-1.66) -- (0,0);
\draw (0.2,0.7) node{(8)};
\draw [very thick, -latex] (0,0) -- (1,-1);
\draw [very thick, -latex] (0,0) -- (1,1);
\end{tikzpicture}
\end{figure}

\begin{figure}[h]
\begin{tikzpicture}
\draw [dotted] (0,-2) grid[xstep=1,ystep=1] (3,2);
\draw (0,0) node{$\bullet$};
\draw (2,0) node{$\bullet$};
\draw (1.6,0.3) node{$2\rho$};
\draw [semithick] (0,0) -- (1.5,1.5) -- (2,1) -- (2.66,-0.34) -- (1.5,-1.5) -- (0,0);
\draw (0.2,0.7) node{(10)};
\draw [very thick, -latex] (0,0) -- (1,-1);
\draw [very thick, -latex] (0,0) -- (1,1);
\end{tikzpicture}
\begin{tikzpicture}
\draw [dotted] (0,-2) grid[xstep=1,ystep=1] (3,2);
\draw (0,0) node{$\bullet$};
\draw (2,0) node{$\bullet$};
\draw (1.6,0.3) node{$2\rho$};
\draw [semithick] (0,0) -- (1.5,1.5) -- (2,1) -- (2.5,0) -- (1.66,-1.66) -- (0,0);
\draw (0.2,0.7) node{(11)};
\draw [very thick, -latex] (0,0) -- (1,-1);
\draw [very thick, -latex] (0,0) -- (1,1);
\end{tikzpicture}
\begin{tikzpicture}
\draw [dotted] (0,-2) grid[xstep=1,ystep=1] (3,2);
\draw (0,0) node{$\bullet$};
\draw (2,0) node{$\bullet$};
\draw (1.6,0.3) node{$2\rho$};
\draw [semithick] (0,0) -- (1.5,1.5) -- (2,1) -- (2.5,0) -- (2,-1) -- (1.5,-1.5) -- (0,0);
\draw (0.2,0.7) node{(12)};
\draw [very thick, -latex] (0,0) -- (1,-1);
\draw [very thick, -latex] (0,0) -- (1,1);
\end{tikzpicture}
\end{figure}

\end{document}